\documentclass[12pt,oneside,leqno]{amsart}
\usepackage{dsfont,amscd}
\usepackage{mathrsfs}
\usepackage{amsmath,amstext,amsthm,amssymb}
\usepackage{typearea}
\usepackage{charter}
\usepackage[T1]{fontenc}
\usepackage{nohyperref}
\usepackage[pageref]{backref}
\renewcommand*{\backref}[1]{}
\renewcommand*{\backrefalt}[4]{%
    \ifcase #1 (Not cited.)%
    \or        ($\uparrow$~#2)
    \else      ($\uparrow$~#2)
    \fi}

\newtheorem{Theorem}{Theorem}[section]
\newtheorem{Proposition}[Theorem]{Proposition}
\newtheorem{Lemma}[Theorem]{Lemma}
\newtheorem{Corollary}[Theorem]{Corollary}
\theoremstyle{definition}
\newtheorem{Definition}[Theorem]{Definition}
\newtheorem{Example}[Theorem]{Example}
\newtheorem{Remark}[Theorem]{Remark}

\DeclareMathOperator{\codim}{codim}

\newcommand{\comment}[1]{}

\begin{document}

\title{Convergence of Fubini-Study currents for orbifold line bundles}
\author{Dan Coman \and George Marinescu}
\thanks{D. Coman is partially supported by the NSF Grant DMS-1300157}
\thanks{G. Marinescu is partially supported by SFB TR 12}
\address{D. Coman: Department of Mathematics, Syracuse
University, Syracuse, NY 13244-1150, USA}
\email{dcoman@syr.edu}
\address{G. Marinescu: Universit\"at zu K\"oln, Mathematisches Institut, Weyertal 86-90, 50931 K\"oln, GERMANY}
\email{gmarines@math.uni-koeln.de}

\pagestyle{myheadings}

\begin{abstract}
We discuss positive closed currents and Fubini-Study currents on orbifolds, as well as Bergman kernels of singular Hermitian orbifold line bundles. We prove that the Fubini-Study currents associated to high powers of a semipositive singular line bundle converge weakly to the curvature current on the set where the curvature is strictly positive, generalizing a well-known theorem of Tian. We include applications to the asymptotic distribution of zeros of random holomorphic sections.
\end{abstract}

\maketitle

\tableofcontents

\section{Introduction}

Let $X$ be a compact complex manifold and $(L,h)$ be a positive holomorphic line bundle over $X$, with Chern curvature form $\omega=c_1(L,h)$. By Kodaira's embedding theorem, high powers $L^p$ give rise
to embeddings
\[\Phi_p:X\to\mathbb{P}(H^0(X,L^p)^*)\,,\quad x\mapsto\big\{S\in H^0(X,L^p):S(x)=0\big\}\]
into projective spaces. The Hermitian metric $h$ on $L$ and the volume form $\omega^n/n!$ on $X$ induce an $L^2$ inner product on $H^0(X,L^p)$, hence an associated Fubini-Study metric $\omega_{FS}$ on $\mathbb{P}(H^0(X,L^p)^*)$.
This gives an \emph{induced Fubini-Study metric} $\gamma_p=\Phi_p^*\omega_{FS}$ on $X$.
The induced Fubini-Study metrics are in some sense ``algebraic'' objects used to approximate the ``transcendental'' metric $\omega$. The relations between these metrics is given by
\begin{equation}\label{e:fs}
\frac1p\gamma_p-\omega=\frac{i}{2\pi p}\,\partial\overline\partial\log P_p\,,
\end{equation}
where $P_p$ is the Bergman kernel function of $H^0(X,L^p)$.	
Following a suggestion of
Yau \cite{Yau87}, Tian \cite[Theorem A]{Ti90} proved that 
\begin{equation}\label{e:tian}
\frac1p\gamma_p\rightarrow\omega\,,\quad\text{as $p\to\infty$, in the $\mathscr{C}^{2}$--\,topology}\,.
\end{equation}
Later, Ruan \cite{Ru98} proved the convergence in the $\mathscr{C}^\infty$--\,topology and improved the estimate of the convergence speed.
In view of \eqref{e:fs}, the proof consists in showing the asymptotics $P_p(x)=p^n+o(p^n)$ in the $\mathscr{C}^k$--\,topology ($k\geq2$), which implies
\begin{equation}\label{e:asy1}
\frac{1}{p}\log P_p\rightarrow 0\,,\quad\text{as $p\to\infty$, in the $\mathscr{C}^{k}$--\,topology}\,.
\end{equation}
Catlin \cite{Ca99}, Zelditch \cite{Z98}, Dai-Liu-Ma \cite{DLM04a}, and \cite{MM04,MM07,MM08} showed the full asymptotics 
\begin{equation}\label{e:full_asy}
P_p(x)=\sum_{r=0}^{\infty}b_r(x)p^{n-r}+O(p^{-\infty})\,,\quad\text{as $p\to\infty$, in the $\mathscr{C}^{k}$--\,topology}\,,
\end{equation} 
which obviously implies \eqref{e:asy1}. The asymptotics \eqref{e:tian} and \eqref{e:full_asy} play an important role
in the Yau-Tian-Donaldson programme of relating the existence of constant scalar
curvature K\"ahler
metrics to a suitable notion of stability, see e.\,g.\ \cite{Don01}, and also in the equidistribution
theorems for zeros of random sections \cite{ShZ99,ShZ04,DS06b,ShZ08,DMS,CM11}.

There are several natural generalizations of these results in the presence of singularities.
One of them is for a compact orbifold $X$ and a positive orbifold line bundle
$L\longrightarrow X$. Then $\gamma_p$ is degenerate at the points with non-trivial isotropy group, that is at points of $X_{sing}^{orb}$\,. However, $\frac1p\gamma_p$ still approximates the original metric on the regular set $X_{reg}^{orb}$\,. Dai-Liu-Ma \cite{DLM04a} (see also \cite[Theorem\,5.4.19]{MM07}) showed that for any $0\leq\alpha<1$ we have $\frac1p\log P_p\rightarrow0$ as $p\to\infty$, in the $\mathscr{C}^{1,\alpha}$--\,topology on local orbifold charts and $\frac1p\gamma_p\rightarrow\omega$ as $p\to\infty$ in the $\mathscr{C}^{k}$--\,topology on compact sets of $X_{reg}^{orb}$\,. 
This is based on the orbifold analogue of (4), see \cite{DLM04a}, \cite[Theorems\,5.4.10\,--11]{MM07} and also \cite{DLM012,RT11} for results on orbifolds with cyclic stabiliser groups.

Another generalization is to consider a smooth manifold $X$ and a smooth line bundle $L$, but
a \emph{singular} Hermitian metric $h$ on $L$ with strictly positive curvature current 
$\omega=c_1(L, h)$. In this case $\gamma_p=\Phi_p^*\omega_{FS}$
are positive currents and it was shown in
\cite{CM11} that the analogue of Tian's result \eqref{e:tian} is
\begin{equation}\label{e:tian2}
\frac1p\gamma_p\rightarrow\omega\,,\quad\text{as $p\to\infty$, weakly in the sense of currents.}
\end{equation}
This follows via \eqref{e:fs} from the fact that $\log P_p$ is locally the difference of two plurisubharmonic functions, thus locally integrable, and
\begin{equation}\label{e:asy2}
\frac{1}{p}\log P_p\rightarrow 0\,,\quad\text{as $p\to\infty$, in $L^1_{loc}(X)$}.
\end{equation}
It turns out that \eqref{e:tian2} and \eqref{e:asy2} are all that is needed to obtain equidistribution results for singular metrics and they are fulfilled in several geometric contexts, see \cite{CM11}.

\par In this paper we consider the following setting:

\medskip

(A) ${\mathcal X}=(X,\mathcal U)$ is a complex (effective) orbifold of dimension $n$, $\widetilde\Omega$ is a Hermitian form on $\mathcal X$ with induced Hermitian form $\Omega$ on the orbifold regular locus of $X$.

\smallskip

(B) $G\subset X$ is an open set with orbifold structure $\mathcal G=(G,\mathcal U_G)$ induced by $\mathcal X$.

\smallskip

(C) $(L,{\mathcal X},\widetilde h)$ is an orbifold line bundle on ${\mathcal X}$ endowed with a singular Hermitian metric $\widetilde h$ with (semi)positive curvature current $c_1(L,\widetilde h)\geq0$.

\medskip

\par We denote by $\widetilde h^p$ the metric induced by $\widetilde h$ on $L^p:=L^{\otimes p}$, and we consider the orbifold canonical bundle $K_{\mathcal X}$ endowed with the metric $\widetilde h^{K_{\mathcal X}}$ induced by the volume form $\widetilde\Omega^n$. With this metric data we can define the following Hilbert spaces of $L^2$-holomorphic sections:

\medskip

(I) $H_{(2)}^0({\mathcal X},L^p\otimes K_{\mathcal X})$ is the space of $L^2$-holomorphic sections of $L^p\otimes K_{\mathcal X}$ endowed with the inner product determined by the metric $\widetilde h^p\otimes\widetilde h^{K_{\mathcal X}}$ and the volume form  $\widetilde\Omega^n$.

\smallskip

(II) $H_{(2)}^0({\mathcal X},L^p)$ is the space of $L^2$-holomorphic sections of $L^p$ endowed with the inner product determined by the metric $\widetilde h^p$ and the volume form  $\widetilde\Omega^n$.

\medskip

\par We refer to Section \ref{S:orb} for background about orbifolds and orbifold line bundles. Let us recall that in the above setting $X$ is a normal complex space and let us denote by $X_{reg}$ the set of its regular points, and by $X_{reg}^{orb}\subset X_{reg}$ its orbifold regular locus (see Section \ref{SS:orb}). In Section \ref{SS:currents} we discuss the notion of orbifold current and show that positive closed orbifold currents of bidegree (1,1) are in one-to-one correspondence to positive closed currents of bidegree (1,1) on $X_{reg}$ (see Propositions \ref{P:pqc} and \ref{P:G}). In view of these, such orbifold currents can be regarded as currents on $X_{reg}$\,. The notion of singular metric on an orbifold line bundle is recalled in Section \ref{SS:singm}, and we refer to Section \ref{SS:FS} for the necessary definitions of spaces of $L^2$-holomorphic sections, their Fubini-Study currents and Bergman kernel functions.

\medskip

\par With these preparations we can state our results.

\begin{Theorem}\label{T:mt1} Let $(L,{\mathcal X},\widetilde h)$, $\mathcal G$, $\widetilde\Omega$ verify assumptions (A)-(C). Suppose that $X_{reg}^{orb}$ carries a complete K\"ahler metric and there exists a continuous function $\varepsilon:G\cap X_{reg}^{orb}\longrightarrow(0,+\infty)$ so that $c_1(L,\widetilde h)>\varepsilon\,\Omega$ on $G\cap X_{reg}^{orb}$\,. If $\gamma_p\,,P_p$ are the Fubini-Study currents, resp.\ the Bergman kernel functions, of the spaces $H_{(2)}^0({\mathcal X},L^p\otimes K_{\mathcal X})$, then:

\par (i) $\frac{1}{p}\,\gamma_p\to c_1(L,\widetilde h)$ weakly as orbifold currents on ${\mathcal G}$ as $p\to\infty$.

\par (ii) $\frac{1}{p}\,\log P_p\to 0$ as $p\to\infty$, in $L^1_{loc}(G)$ with respect to the area measure on $X$.

\end{Theorem}

\par Note that the metric $\Omega$ is not assumed to be complete, we assume just that $X_{reg}^{orb}$ carries \emph{some} complete K\"ahler metric. The assumption that $X_{reg}^{orb}$ is K\"ahler complete is known to hold in the following situations, thanks to an argument in \cite{O87}:

\begin{Proposition}\label{P:cKS} $X_{reg}^{orb}$ admits a complete K\"ahler metric if one of the following three conditions hold:

(i) $({\mathcal X},\widetilde\Omega)$ is a compact K\"ahler orbifold, or

(ii) $X$ is a Stein space, or

(iii) ${\mathcal X}$ is a complete K\"ahler manifold.
\end{Proposition}
An immediate consequence of Theorem \ref{T:mt1} for $L=K_{\mathcal X}$ is the following.

\begin{Corollary}\label{C:mt1} Let $\mathcal X$ be an orbifold so that $X_{reg}^{orb}$ carries a complete K\"ahler metric, $G\subset X$ an open set, and $\widetilde\Omega$ a Hermitian metric on $\mathcal X$ so that $c_1(K_{\mathcal X},\widetilde h)\geq0$, where $\widetilde h=\widetilde h^{K_{\mathcal X}}$ is the metric on $K_{\mathcal X}$ induced by $\widetilde\Omega$. Suppose that there exists a continuous function $\varepsilon:G\cap X_{reg}^{orb}\longrightarrow(0,+\infty)$ so that $c_1(K_{\mathcal X},\widetilde h)>\varepsilon\,\Omega$ on $G\cap X_{reg}^{orb}$. Then $\frac{1}{p}\,\gamma_p\to c_1(K_{\mathcal X},\widetilde h)$ weakly on ${\mathcal G}$, and $\frac{1}{p}\,\log P_p\to 0$  in $L^1_{loc}(G)$, as $p\to\infty$, where $\gamma_p\,,P_p$ are the Fubini-Study currents, resp. the Bergman kernel functions, of $H_{(2)}^0({\mathcal X},K_{\mathcal X}^p)$.
\end{Corollary}

\par Under stronger hypotheses, we can formulate a version of Theorem \ref{T:mt1} for $L^2$-holomorphic  sections of $L^p$ rather than $L^p$-valued holomorphic $n$-forms:


\begin{Theorem}\label{T:mt2} Let $(L,{\mathcal X},\widetilde h)$, $\mathcal G$, $\widetilde\Omega$ verify assumptions (A)-(C) such that $X_{reg}^{orb}$ carries a complete K\"ahler metric and $\widetilde\Omega$ is a K\"ahler form with semipositive Ricci form $Ric_{\widetilde\Omega}\geq0$. If there exists a continuous function $\varepsilon:G\cap X_{reg}^{orb}\longrightarrow(0,+\infty)$ so that $c_1(L,\widetilde h)>\varepsilon\,\Omega$ on $G\cap X_{reg}^{orb}$, then the conclusions (i)-(ii) of Theorem \ref{T:mt1} hold for the Fubini-Study currents $\gamma_p$ and the Bergman kernel functions $P_p$ of the spaces $H_{(2)}^0({\mathcal X},L^p)$.
\end{Theorem}

\par For overall positive holomorphic orbifold line bundles with smooth metrics we have the following result.

\begin{Theorem}\label{T:mt3} Let $({\mathcal X},\widetilde\Omega)$ be a compact Hermitian orbifold and $(L,{\mathcal X},\widetilde h)$ be a positive orbifold line bundle endowed with a smooth positively curved metric $\widetilde{h}$. Then $\frac{1}{p}\,\gamma_p\to c_1(L,\widetilde h)$ weakly on ${\mathcal X}$ as $p\to\infty$, where $\gamma_p$ are the Fubini-Study currents of $H_{(2)}^0({\mathcal X},L^p)$.
\end{Theorem}

\medskip

\par An important application of Theorems \ref{T:mt1}, \ref{T:mt2} and \ref{T:mt3} is to the study of the asymptotic distribution of zeros of random holomorphic sections. After the pioneering work of Nonnenmacher-Voros
\cite{NoVo:98}, general methods were developed by Shiffman-Zelditch \cite{ShZ99,ShZ04,ShZ08} and Dinh-Sibony \cite{DS06b} to describe the asymptotic distribution of zeros of random holomorphic sections of a positive line bundle over a projective manifold endowed with a smooth positively curved metric. The paper \cite{DS06b} gives moreover very good convergence speed
and applies to general measures (e.\,g.\ equidistribution of complex zeros of homogeneous polynomials with real coefficients). Some important technical tools for higher dimension used in the previous works were introduced
by Forn{\ae}ss-Sibony \cite{FS95}. 
For the non-compact setting and the case of singular Hermitian metrics see \cite{DMS,CM11}.


\par Suppose that we are in either one of the settings of Theorem \ref{T:mt1}, or Corollary \ref{C:mt1}, or Theorem \ref{T:mt2}, or Theorem \ref{T:mt3}. We assume in addition that  $\mathcal X$ is compact and $\widetilde\Omega$ is a K\"ahler form and we denote by $\mathcal V^p$ the corresponding spaces of $L^2$-holomorphic sections in each of the above settings.

\par 
We let $\lambda_p$ be the normalized surface measure on the unit sphere ${\mathcal S}^p$ of $\mathcal V^p$, defined in the natural way by using an orthonormal basis of $\mathcal V^p$ (see Section \ref{S:SZ}). Consider the probability space ${\mathcal S}_\infty=\prod_{p=1}^\infty{\mathcal S}^p$ endowed with the probability measure $\lambda_\infty=\prod_{p=1}^\infty\lambda_p$. We have the following theorem:

\begin{Theorem}\label{T:SZ} In either one of the settings of Theorem \ref {T:mt1}, \ref{T:mt2}, \ref{T:mt3}, or Corollary \ref{C:mt1}, assume in addition that $\mathcal X$ is compact and $\widetilde\Omega$ is a K\"ahler form. Then:

\par (i) The Fubini-Study current of $\mathcal V^p$ is the expectation $E_p[S=0]$ of the current-valued random variable $S\in\mathcal S^p\to[S=0]$, given by
$$\langle\,E_p[S=0],\widetilde\theta\;\rangle=\int_{{\mathcal S}^p}\langle\,[S=0],\widetilde\theta\;\rangle\,d\lambda_p(S)\,,$$
where $\widetilde\theta$ is a test form on $\mathcal X$. We have that $\frac{1}{p}\,E_p[S=0]\to c_1(L,\widetilde h)$ as $p\to\infty$, weakly as orbifold currents on $\mathcal G$ (resp. on $\mathcal X$, in the case of Theorem \ref{T:mt3}).

\par (ii) For $\lambda_\infty$-a.e. sequence $\{\sigma_p\}_{p\geq1}\in{\mathcal S}_\infty\,$, we have that $\frac{1}{p}\,[\sigma_p=0]\to c_1(L,\widetilde h)$ as $p\to\infty$, weakly as orbifold currents on $\mathcal G$ (resp. on $\mathcal X$, in the case of Theorem \ref{T:mt3}).
\end{Theorem}

\par Here $[S=0]$ denotes the current of integration (with multiplicities) over the zero set of a nontrivial section $S\in\mathcal V^p$. The arguments of Shiffman and Zelditch \cite{ShZ99} needed for the proof of Theorem \ref{T:SZ} are recalled in Section \ref{S:SZ}.

\medskip

\par {\em Acknowledgement.} Dan Coman is grateful to the Alexander von Humboldt Foundation for their support and to the Mathematics Institute at the University of K\"oln for their hospitality.

\section{Orbifolds and orbifold line bundles}\label{S:orb}

We recall here some necessary notions about (complex effective) orbifolds and (holomorphic) orbifold line bundles, following \cite{BG08} (see also \cite{ALR07,BGK05,GK07,MM07}).

\subsection{Orbifolds}\label{SS:orb}

\par Let $X$ be a (second countable) complex space of dimension $n$. An orbifold chart on $X$ is a triple $(\widetilde U, \Gamma, \phi)$ where $\widetilde U$ is a domain in ${\mathbb C}^n$, $\Gamma $ is a finite group acting effectively as automorphisms of $\widetilde U$, and $\phi:\widetilde U\longrightarrow U$ is an analytic cover (i.e. proper and finite holomorphic map) onto an open set $U\subset X$ such that $\phi\circ\gamma=\phi$ for every $\gamma\in\Gamma$ and the induced natural map $\widetilde U/\Gamma\longrightarrow U$ is a homeomorphism. An injection between two charts $(\widetilde U, \Gamma, \phi)$, $(\widetilde U', \Gamma', \phi')$ is a holomorphic embedding $\lambda:\widetilde U\longrightarrow\widetilde U'$ so that $\phi'\circ\lambda=\phi$. An orbifold atlas on $X$ is a family ${\mathcal U} = \{(\widetilde U_i, \Gamma_i, \phi_i)\}$ of orbifold charts such that $X = \bigcup U_i$, where $U_i:=\phi_i(\widetilde U_i)$, and, given two charts $(\widetilde U_i, \Gamma_i, \phi_i)$, $(\widetilde U_j, \Gamma_j, \phi_j)$ and $x\in U_i\cap U_j$, there exist a chart $(\widetilde U_k, \Gamma_k, \phi_k)$ with $x\in U_k$ and injections $\lambda_{ik}: (\widetilde U_k, \Gamma_k, \phi_k)\longrightarrow(\widetilde U_i, \Gamma_i, \phi_i)$, $\lambda_{jk}: (\widetilde U_k, \Gamma_k, \phi_k)\longrightarrow(\widetilde U_j, \Gamma_j, \phi_j)$. An atlas ${\mathcal U}$ is said to be a refinement of an atlas ${\mathcal V}$ if there exists an injection of every chart of ${\mathcal U}$ into some chart of ${\mathcal V}$. An orbifold ${\mathcal X}=(X,{\mathcal U})$ is a complex space $X$ with a (maximal) orbifold atlas ${\mathcal U}$. It follows that the underlying space $X$ is a reduced normal complex space with at most quotient singularities (see e.g. \cite[Sec. 4.4]{BG08}).

\par Given an injection $\lambda:\widetilde U\longrightarrow\widetilde U'$ and $\gamma\in\Gamma$, one  has that there exists a unique $\gamma'\in\Gamma'$ with $\gamma'\circ\lambda=\lambda\circ\gamma$ \cite[Lemma 4.1.2]{BG08}. Thus we get an injective group homomorphism, denoted still by $\lambda:\Gamma\longrightarrow\Gamma'$, defined by $\lambda(\gamma)=\gamma'$. Moreover, if $\gamma'\lambda(\widetilde U)\cap\lambda(\widetilde U)\neq\emptyset$ then $\gamma'\in\lambda(\Gamma)$, and so $\gamma'\lambda(\widetilde U)=\lambda(\widetilde U)$ \cite[p. 3]{ALR07}. This implies that the set $\gamma'\lambda(\widetilde U)$ depends only on the coset $[\gamma']=\gamma'\,\lambda(\Gamma)$ and
\begin{equation}\label{e:coset}
(\phi')^{-1}(U)=\bigcup_{[\gamma']}\gamma'\lambda(\widetilde U)\,,
\end{equation}
where the union is disjoint.

\par We write
$$X=X_{reg}\cup X_{sing},$$
where $X_{reg}$ (resp.\ $X_{sing}$) is the set of regular (resp.\ singular) points of $X$. Since $X$ is normal, we have that $X_{reg}$ is a connected complex manifold and $X_{sing}$ is a closed reduced complex subspace of $X$ with $\codim X_{sing}\geq2$. Given an orbifold chart $(\widetilde U, \Gamma, \phi)$, the isotropy group $\Gamma_x$ of $x\in U$ is defined as the isotropy (stabilizer) group $\Gamma_y$ of any $y\in\phi^{-1}(x)$, which is unique up to conjugacy. The sets of orbifold regular, resp. orbifold singular, points are defined by
$$X_{reg}^{orb}=\{x\in X:\,|\Gamma_x|=1\}\,,\;X_{sing}^{orb}=\{x\in X:\,|\Gamma_x|>1\}\,.$$
Then $X_{sing}^{orb}$ is a closed complex subspace of $X$ and one has that $X_{reg}^{orb}\subset X_{reg}$ and $X_{sing}\subset X_{sing}^{orb}$.

\par An orbifold ${\mathcal X}=(X,{\mathcal U})$ can be identified with the log pair $(X,\Delta)$, where $\Delta$ is called the branch divisor and is the ${\mathbb Q}$-divisor defined by
$$\Delta=\sum\left(1-m^{-1}\right)\,D\,.$$
Here the sum is taken over all Weil divisors $D\subset X_{sing}^{orb}$ and $m$ is the ramification index over $D$ (see \cite{BGK05,GK07}, \cite[Sec. 4.4]{BG08}).

\subsection{Orbifold line bundles}\label{SS:olb}

\par We now recall the notion of a (proper) orbifold line bundle on ${\mathcal X}=(X,{\mathcal U})$ (see \cite{BG08}). This is a collection $\{L_{\widetilde U_i}\}$ of $\Gamma_i$-equivariant holomorphic line bundles $\widetilde\pi_i:L_{\widetilde U_i}\longrightarrow\widetilde U_i$ which satisfy a gluing condition. Equivariance means that $\Gamma_i$ acts effectively on $L_{\widetilde U_i}$ as bundle maps which are isomorphisms along fibers and the following diagram is commutative,
\begin{equation*}
\begin{CD}
L_{\widetilde U_i}@>\widetilde\gamma>>L_{\widetilde U_i}\\
@VV \widetilde\pi_i V    @VV \widetilde\pi_i V \\
\widetilde U_i@>\gamma>>\widetilde U_i
\end{CD}\;\;\;,\
\end{equation*}
where $\widetilde\gamma$ is the bundle map corresponding to $\gamma\in\Gamma_i$. The gluing condition is as follows: any injection $\lambda_{ji}:\widetilde U_i\longrightarrow\widetilde U_j$ induces a bundle map $\widetilde\lambda_{ji}:L_{\widetilde U_j}\mid_{_{\lambda_{ji}(\widetilde U_i)}}\longrightarrow L_{\widetilde U_i}$ which is an isomorphism along fibers, so that if $\gamma\in\Gamma_i$ and $\gamma'=\lambda_{ji}(\gamma)\in\Gamma_j$ (i.e. $\gamma'\circ\lambda_{ji}=\lambda_{ji}\circ\gamma$) then $\widetilde\lambda_{ji}\circ\widetilde\gamma'=\widetilde\gamma\circ\widetilde\lambda_{ji}$. Moreover, the $\widetilde\lambda_{ji}$ are functorial: if $\lambda_{kj}:\widetilde U_j\longrightarrow\widetilde U_k$ is another injection then $\widetilde{\lambda_{kj}\circ\lambda_{ji}}=\widetilde\lambda_{ji}\circ\widetilde\lambda_{kj}$.

\par The total space $L$ of the orbifold line bundle $\{L_{\widetilde U_i}\}$ is constructed by gluing the sets $L_{\widetilde U_i}/\Gamma_i$ via the maps $\widetilde\lambda_{ji}$ in the usual way. Let $\widetilde\phi_i:L_{\widetilde U_i}\longrightarrow L_{\widetilde U_i}/\Gamma_i$ be the natural map. Then the atlas $\{(L_{\widetilde U_i},\Gamma_i,\widetilde\phi_i)\}$ gives $L$ the structure of an  (effective) orbifold, hence $L$ is a normal complex space. Since
$$\phi_i\circ\widetilde\pi_i\circ\widetilde\gamma=\phi_i\circ\gamma\circ\widetilde\pi_i=\phi_i\circ\widetilde\pi_i\,,\;\forall\,\gamma\in\Gamma_i\,,$$
there exists a continuous map $\pi_i$ which makes the following diagram commutative:
\begin{equation*}
\begin{CD}
L_{\widetilde U_i}@>\widetilde\phi_i>>L_{\widetilde U_i}/\Gamma_i\\
@VV \widetilde\pi_i V    @VV \pi_i V \\
\widetilde U_i@>\phi_i>>U_i
\end{CD}\;\;\;.\
\end{equation*}
Note that $\pi_i$ is surjective and is holomorphic on $\pi_i^{-1}(U_i\cap X_{reg}^{orb})$, hence it is holomorphic. The maps $\pi_i$ glue to a surjective holomorphic map $\pi:L\longrightarrow X$.

\par For brevity, we will denote the orbifold line bundle $\{L_{\widetilde U_i}\}$ on ${\mathcal X}=(X,{\mathcal U})$ by $(L,{\mathcal X})$. If $L\mid_{X_{reg}^{orb}}:=\pi^{-1}(X_{reg}^{orb})$, then $\pi:L\mid_{X_{reg}^{orb}}\longrightarrow X_{reg}^{orb}$ is a holomorphic line bundle. Suppose now that $x\in X$ has non-trivial isotropy group. If $y\in\phi^{-1}(x)$ for some orbifold chart $(\widetilde U,\Gamma,\phi)$ near $x$, then each $\gamma\in\Gamma_y$ induces an isomorphism $\widetilde\gamma:L_{\widetilde U}\mid_y\longrightarrow L_{\widetilde U}\mid_y$. Hence $L\mid_x:=\pi^{-1}(x)\cong(L_{\widetilde U}\mid_y)/\Gamma_y$, and $\pi:L\longrightarrow X$ is in general not a holomorphic line bundle in the usual sense. The latter are sometimes called absolute orbifold line bundles, see \cite{BG08}. However if $X$ is compact then there exists $m\geq1$ so that, for every orbifold line bundle $L$ on ${\mathcal X}$, $L^{\otimes m}$ is absolute.

\par A (holomorphic) section of $(L,{\mathcal X})$ is a collection of sections $\widetilde S_i:\widetilde U_i\longrightarrow L_{\widetilde U_i}$ for each orbifold chart $(\widetilde U_i, \Gamma_i, \phi_i)$ so that $\widetilde S_i$ is $\Gamma_i$-equivariant, i.e. $\widetilde\gamma\circ\widetilde S_i=\widetilde S_i\circ\gamma$ for all $\gamma\in\Gamma_i$, and for every injection $\lambda_{ji}:\widetilde U_i\longrightarrow\widetilde U_j$ one has $ \widetilde\lambda_{ji}\circ\widetilde S_j\circ\lambda_{ji}=\widetilde S_i$.

\par Since $\widetilde\phi_i\circ\widetilde S_i\circ\gamma=\widetilde\phi_i\circ\widetilde\gamma\circ\widetilde S_i=\widetilde\phi_i\circ\widetilde S_i$ for each $\gamma\in\Gamma_i$, there exists a continuous map $S_i$ which makes the following diagram commutative:
\begin{equation*}
\begin{CD}
L_{\widetilde U_i}@>\widetilde\phi_i>>L\mid_{U_i}\\
@AA \widetilde S_i A    @AA S_i A \\
\widetilde U_i@>\phi_i>>U_i
\end{CD}\;\;\;.\
\end{equation*}
Note that $S_i$ is in fact holomorphic on $U_i$, since it is holomorphic on $U_i\cap X_{reg}^{orb}$. The local sections $S_i$ glue to an injective holomorphic map $S:X\longrightarrow L$ which verifies $\pi\circ S=id_X$. Its restriction to $X_{reg}^{orb}$ is a section of the line bundle $L\mid_{X_{reg}^{orb}}$\,.

\par We denote by $H^0({\mathcal X},L)$ the vector space of holomorphic sections of $(L,{\mathcal X})$.

\subsection{Differential forms on orbifolds}\label{SS:dfo}

\par A $(p,q)$ form on an orbifold ${\mathcal X}=(X,{\mathcal U})$ is a collection $\widetilde\psi=\{\widetilde\psi_i\}$ of smooth $(p,q)$ forms $\widetilde\psi_i$ on $\widetilde U_i$, for each orbifold chart $(\widetilde U_i, \Gamma_i, \phi_i)$, so that $\gamma^\star\widetilde\psi_i=\widetilde\psi_i$ for each $\gamma\in\Gamma_i$, and $\lambda_{ji}^\star\widetilde\psi_j=\widetilde\psi_i$ for each injection $\lambda_{ji}:\widetilde U_i\longrightarrow\widetilde U_j$. Equivalently, a $(p,q)$ form on ${\mathcal X}$ is a smooth $(p,q)$ form $\psi$ on $X_{reg}^{orb}$ so that for each chart $(\widetilde U_i, \Gamma_i, \phi_i)$, $\phi_i^\star(\psi\mid_{U_i\cap X_{reg}^{orb}})$ extends to a smooth form on $\widetilde U_i$ \cite{BGK05}.

\par A $(p,p)$ form $\widetilde\psi=\{\widetilde\psi_i\}$ is positive (resp. closed) if each form $\widetilde\psi_i$ is positive (resp. closed). A $(1,1)$ form $\widetilde\Omega=\{\widetilde\Omega_i\}$ is K\"ahler, resp. Hermitian, if each form $\widetilde\Omega_i$ is K\"ahler, resp. Hermitian (i.e. positive).

\par One can define the orbifold tangent and cotangent bundles of ${\mathcal X}$ and view orbifold differential forms as sections of such. An important example that we will need is the orbifold canonical bundle $K_{\mathcal X}$, defined as the collection of canonical bundles $\{K_{\widetilde U_i}\}$, for all charts $(\widetilde U_i, \Gamma_i, \phi_i)\in{\mathcal U}$. Note that the equivariance and gluing are given by the pull-back operators, $\widetilde\gamma:=(\gamma^{-1})^\star$, $\widetilde\lambda_{ji}:=\lambda_{ji}^\star$.

\section{Currents and singular metrics}\label{S:csm}

\par We collect here a few facts about currents on orbifolds, being especially interested in positive closed currents of bidegree (1,1). We also recall the notion of singular metric on an orbifold line bundle and we introduce the Bergman kernel function and the Fubini-Study currents for subspaces of $L^2$-holomorphic sections. Throughout this section, we use the notations introduced in Section \ref{S:orb}.

\subsection{Currents on orbifolds}\label{SS:currents}

\par A current of bidegree $(p,q)$ on an orbifold ${\mathcal X}=(X,{\mathcal U})$ is a collection $\widetilde T=\{\widetilde T_i\}$ of currents $\widetilde T_i$ of bidegree $(p,q)$ on $\widetilde U_i$, for each orbifold chart $(\widetilde U_i, \Gamma_i, \phi_i)$, so that $\gamma_\star \widetilde T_i=\widetilde T_i$ for each $\gamma\in\Gamma_i$, and $(\lambda_{ji})_\star \widetilde T_i=\widetilde T_j\mid_{\lambda_{ji}(\widetilde U_i)}$ for each injection $\lambda_{ji}:\widetilde U_i\longrightarrow\widetilde U_j$. A current $\widetilde T$ of bidegree $(p,p)$ is positive (resp.\ closed) if each current $\widetilde T_i$ is positive (resp.\ closed). A sequence of currents $\widetilde T^k=\{\widetilde T^k_i\}$, $k\geq1$, converges weakly to a current $\widetilde T=\{\widetilde T_i\}$ if, for each $i$, the sequence of current $\widetilde T^k_i$ converges weakly on $\widetilde U_i$ to $\widetilde T_i$ as $k\to\infty$.

\par Let us define the action of an orbifold current $\widetilde T=\{\widetilde T_i\}$ of bidegree $(p,q)$ on a $(n-p,n-q)$ test form $\widetilde\theta=\{\widetilde\theta_i\}$ (where ${\rm supp}\,\widetilde\theta\Subset X$). Fix a partition of unity $\{\chi_l\}_{l\geq1}$ on $X$ so that $\chi_l$ has compact support contained in $U_{i_l}$ and set
$$\big\langle\,\widetilde T,\widetilde\theta\,\big\rangle=\sum_{l=1}^\infty\frac{1}{m_{i_l}}\,\langle\widetilde T_{i_l},(\chi_l\circ\phi_{i_l})\,\widetilde\theta_{i_l}\rangle\,,$$
where $m_i=|\Gamma_i|$. Note that the standard calculus with currents works as usual. For instance, one checks that the current $d\widetilde T:=\{d\widetilde T_i\}$ verifies $\langle d\widetilde T,\widetilde\theta\rangle=(-1)^{p+q+1}\langle\widetilde T,d\widetilde\theta\rangle$.

\par If $U_i\subset X_{reg}$ the current $(\phi_i)_\star \widetilde T_i$ is well defined on $U_i$, since $\phi_i$ is proper. We show that these currents glue to a global current on $X_{reg}$:

\begin{Proposition}\label{P:pqc} Let $\widetilde T=\{\widetilde T_i\}$ be a current of bidegree $(p,q)$ on an orbifold ${\mathcal X}=(X,{\mathcal U})$. There exists a current $T$ of bidegree $(p,q)$ on $X_{reg}$ so that for every chart $(\widetilde U_i, \Gamma_i, \phi_i)$ with $U_i\subset X_{reg}$ we have $T\mid_{U_i}=\frac{1}{m_i}\,(\phi_i)_\star \widetilde T_i\,$, where $m_i=|\Gamma_i|$.
\end{Proposition}

\begin{proof} Let $T_i:=\frac{1}{m_i}\,(\phi_i)_\star \widetilde T_i$. If $\lambda_{ji}:\widetilde U_i\longrightarrow\widetilde U_j$ is an injection with $U_i\subset U_j\subset X_{reg}$ we have to show that $T_j\mid_{U_i}=T_i$. If $V:=\lambda_{ji}(\widetilde U_i)$ we have by \eqref{e:coset} that
$$\phi_j^{-1}(U_i)=\bigcup_{[\gamma']}\gamma'\,V\,,$$
where the union is disjoint and $[\gamma']=\gamma'\lambda_{ji}(\Gamma_i)$ denotes the coset of $\gamma'$ in $\Gamma_j/\lambda_{ji}(\Gamma_i)$. This implies that the map $\widehat\phi_j:=\phi_j\mid_V:V\longrightarrow U_i$ is an analytic cover with topological degree $m_i=|\lambda_{ji}(\Gamma_i)|$ and
$$T_i=\frac{1}{m_i}\,(\phi_i)_\star \widetilde T_i=\frac{1}{m_i}\,(\widehat \phi_j)_\star\circ(\lambda_{ji})_\star \widetilde T_i=\frac{1}{m_i}\,(\widehat\phi_j)_\star(\widetilde T_j\mid_V)\,.$$
If $\theta$ is an $(n-p,n-q)$ test form supported in $U_i$ then
\begin{eqnarray*}
\big\langle\, T_j\mid_{U_i},\theta\big\rangle&=&\frac{1}{m_j}\,\big\langle\, \widetilde T_j,\phi_j^\star\theta\big\rangle=\frac{1}{m_j}\,\sum_{[\gamma']}\big\langle\, \widetilde T_j\mid_{\gamma'\,V},(\phi_j^\star\theta)\mid_{\gamma'\,V}\big\rangle\\&=&\frac{1}{m_j}\,\sum_{[\gamma']}\big\langle\,\gamma'_\star(\widetilde T_j\mid_V),(\phi_j^\star\theta)\mid_{\gamma'\,V}\rangle=\frac{1}{m_j}\,\sum_{[\gamma']}\big\langle\, \widetilde T_j\mid_V,\widehat\phi_j^\star\theta\rangle=\big\langle\, T_i,\theta\big\rangle\,.
\end{eqnarray*}
\end{proof}

\par Let $\widetilde{\mathcal T}$, resp.\ $\mathcal T$, denote the set of positive closed currents of bidegree $(1,1)$ on ${\mathcal X}$, resp.\ on $X_{reg}$\,. Proposition \ref{P:pqc} provides a map $\boldsymbol{F}:\widetilde{\mathcal T}\longrightarrow\mathcal T$, $\boldsymbol{F}({\widetilde T})=T$, which is continuous with respect to weak$^\star$ convergence of currents. We will prove that $\boldsymbol{F}$ is in fact bijective and has continuous inverse. For this, we show first the following:

\begin{Proposition}\label{P:11c} Let ${\mathcal X}=(X,{\mathcal U})$ be an orbifold. Then each point $x\in X$ has a neighborhood $U\subset X$ such that for every current $T\in{\mathcal T}$ there is a psh function $v$ on $U$ with $dd^cv=T$ on $U\cap X_{reg}$\,.
\end{Proposition}

\begin{proof} 
Recall that a function $v:U\to[-\infty,\infty)$ on a complex space $U$ is called plurisubharmonic (abbreviated psh),
if $v$ is not identically $-\infty$ in any open set of $U$ and for every $x\in U$ there is a neighbourhood $V$ of $x$, and an analytic isomorphism $\iota$
of $V$ onto an analytic set in a polydisc $P$ in some $\mathbb{C}^N$ such that $v =\widetilde{v}\circ\iota$ in $V$ for some psh function $\widetilde{v}$ on $P$. The space of psh functions on $U$ will be denoted by $PSH(U)$. If $\iota$, $P$ and $\widetilde{v}$ can be so chosen that $\widetilde{v}$ is strictly plurisubharmonic in $P$ then we call $v$ strictly plurisubharmonic in $U$.

The conclusion of the Proposition is clear if $x\in X_{reg}$, so we assume $x\in X_{sing}$\,. Fix an orbifold chart $(\widetilde U_i, \Gamma_i, \phi_i)$ with $x\in U_i=\phi_i(\widetilde U_i)$ and set $m_i=|\Gamma_i|$, $\Sigma=U_i\cap X_{sing}$\,. Let $V\subset \widetilde U_i$ be an open set so that all its connected components are simply connected and $\phi_i^{-1}(x)\subset V$. Since $\phi_i$ is proper and $X$ is locally irreducible, there exists a neighborhood $U\subset U_i$ of $x$ so that $U\setminus\Sigma$ is connected and $\phi_i^{-1}(U)\subset V$.

\par Let $T\in{\mathcal T}$. Since $\codim\Sigma\geq2$ and $\phi_i$ is finite we have $\codim\phi_i^{-1}(\Sigma)\geq2$. Hence the positive closed $(1,1)$ current $R=\phi_i^\star(T\mid_{U_i\setminus\Sigma})$ on $\widetilde U_i\setminus\phi_i^{-1}(\Sigma)$ extends to a positive closed current on $\widetilde U_i$. Moreover $\gamma^\star R=R$ on $\widetilde U_i$. Indeed, since $\phi_i\circ\gamma=\phi_i$ this holds on the set $\widetilde U_i\setminus\phi_i^{-1}(\Sigma)$ which is $\Gamma_i$-invariant, hence it holds on $\widetilde U_i$. By the assumption on $V$, we have $R=dd^cu'$ for some psh function $u'$ on $V$. As $\phi_i^{-1}(U)$ is $\Gamma_i$-invariant we can define
$$u=\frac{1}{m_i}\,\sum_{\gamma\in\Gamma_i}u'\circ\gamma\in PSH(\phi_i^{-1}(U))\,,\;\;{\rm so}\;\;dd^cu=\frac{1}{m_i}\,\sum_{\gamma\in\Gamma_i}\gamma^\star R=R\,.$$

\par Note that $u\circ\gamma=u$ for all $\gamma\in\Gamma_i$. Thus $u=v\circ\phi_i$ for some upper semicontinuous function $v$ on $U$. As $\phi_i:\phi_i^{-1}(U\setminus\Sigma)\longrightarrow U\setminus\Sigma$ is proper and $v(y)=u(z)$ for any $z\in\phi_i^{-1}(y)$, it follows that $v$ is psh on $U\setminus\Sigma$. Since $v$ is upper semicontinuous, we have $v\in PSH(U)$ \cite[Theorem 1.7]{D85} .

\par If $S$ is a positive closed (1,1) current on $U\setminus\Sigma$ then $(\phi_i)_\star\circ\phi_i^\star S=m_iS$. Indeed, if $\theta$ is a test form supported in $U\setminus\Sigma$ we may assume that $S=dd^c\rho$ for a psh function $\rho$ near the support of $\theta$, and
\begin{equation}\label{e:pushpull}
\big\langle\,(\phi_i)_\star\circ\phi_i^\star S,\theta\,\big\rangle=\big\langle\, dd^c(\rho\circ\phi_i),\phi_i^\star\theta\,\big\rangle=m_i\big\langle\, dd^c\rho,\theta\,\big\rangle=m_i\big\langle\, S,\theta\,\big\rangle\,.
\end{equation}
Since $\phi_i^\star T=R=dd^c(v\circ\phi_i)=\phi_i^\star(dd^cv)$ on $\phi_i^{-1}(U\setminus\Sigma)$, we deduce $T=dd^cv$ on $U\setminus\Sigma$\,.
\end{proof}

\begin{Proposition}\label{P:G} The map $\boldsymbol{G}:{\mathcal T}\longrightarrow\widetilde{\mathcal T}$, $\boldsymbol{G}(T)=\{\phi_i^\star(T\mid_{U_i})\}$, where $(\widetilde U_i, \Gamma_i, \phi_i)$ are the orbifold charts of $\mathcal X$, is well defined, continuous with respect to weak$^\star$ convergence, and it is the inverse of $\boldsymbol{F}$.
\end{Proposition}

\begin{proof} We can define $\widetilde T_i:=\phi_i^\star(T\mid_{U_i})$ as the positive closed (1,1) current on $\widetilde U_i$ with local potentials $v\circ\phi_i$, where $v$ are the local potentials of $T$ near each point of $U_i$ provided by Proposition \ref{P:11c}. Clearly, $\gamma^\star\widetilde T_i=\widetilde T_i$ for all $\gamma\in\Gamma_i$. If $\lambda_{ji}:\widetilde U_i\longrightarrow\widetilde U_j$ is an injection then, with the notation $\widehat\phi_j$ from the proof of Proposition \ref{P:pqc},
$$\lambda_{ji}^\star(\widetilde T_j\mid_{\lambda_{ji}(\widetilde U_i)})=\lambda_{ji}^\star((\phi_j^\star T)\mid_{\lambda_{ji}(\widetilde U_i)})=\lambda_{ji}^\star(\widehat\phi_j^\star(T\mid_{U_i}))=\phi_i^\star(T\mid_{U_i})=\widetilde T_i\,.$$
Hence $\boldsymbol{G}(T)\in\widetilde{\mathcal T}$. By \eqref{e:pushpull}, $(\phi_i)_\star\widetilde T_i=m_iT\mid_{U_i}$, where $m_i=|\Gamma_i|$, so $\boldsymbol{F}\circ \boldsymbol{G}(T)=T$.

\par We note that $\boldsymbol{G}$ is surjective. Indeed, if $\widetilde T=\{\widetilde T_i\}\in\widetilde{\mathcal T}$ and $x\in U_i$, we can repeat the argument in the proof of Proposition \ref{P:11c} to show that there exists a small neighborhood $U\subset U_i$ of $x$ and $v\in PSH(U)$ so that $\widetilde T_i=dd^c(v\circ\phi_i)$ on $\phi^{-1}(U)$ and $\frac{1}{m_i}\,(\phi_i)_\star\widetilde T_i=dd^cv$ on $U\cap X_{reg}$. Setting $T:=\boldsymbol{F}(\widetilde T)$ we have for $U_i\subset X_{reg}$, $T\mid_U=\frac{1}{m_i}\,(\phi_i)_\star\widetilde T_i\mid_U=dd^cv$, so $\widetilde T=\boldsymbol{G}(T)$.

\par To prove the continuity of $\boldsymbol{G}$, assume that $T^j,\,T\in{\mathcal T}$ and the sequence $T^j$ converges weakly to $T$. Fix an orbifold chart $(\widetilde U_i, \Gamma_i, \phi_i)$ and let $\Sigma=U_i\cap X_{sing}$. Then $\phi_i^\star T^j$ converges weakly to $\phi_i^\star T$ on $\widetilde U_i\setminus\phi_i^{-1}(\Sigma)$. Since $\codim\phi_i^{-1}(\Sigma)\geq2$, Oka's inequality for currents \cite{FS95} implies that the sequence of currents $\{\phi_i^\star T^j\}_j$ has locally bounded mass in $\widetilde U_i$. As any limit point equals $\phi_i^\star T$ on $\widetilde U_i\setminus\phi_i^{-1}(\Sigma)$, we conclude that the sequence $\phi_i^\star T^j$ converges weakly to $\phi_i^\star T$ on $\widetilde U_i$.
\end{proof}

\begin{Definition}\label{D:Kc} A K\"ahler current on $X$ is a positive closed $(1,1)$ current $T$ on $X_{reg}$ with the property that for every $x\in X$ there exist a neighborhood $U$ of $x$ and a strictly psh function $v$ on $U$ so that $T=dd^cv$ on $U\cap X_{reg}$.
\end{Definition}

\par We note that if the local potentials of a K\"ahler current $T$ are $C^\infty$-smooth then $X$ is called a K\"ahler space, cf. \cite{O87} (see also \cite[Sec. 5.2]{EGZ09}). Propositions \ref{P:11c} and \ref{P:G} yield the following:

\begin{Proposition}\label{P:Kc} If ${\mathcal X}=(X,{\mathcal U})$ is a K\"ahler orbifold then $X$ carries a K\"ahler current whose local potentials are continuous near each $x\in X$ and smooth near each $x\in X_{reg}^{orb}$\,.
\end{Proposition}

\begin{proof} Let $\widetilde\Omega=\{\widetilde\Omega_i\}$ be a K\"ahler form on $\mathcal X$ and let $\Omega=\boldsymbol{F}(\widetilde\Omega)\in{\mathcal T}$. Proposition \ref{P:G} and its proof shows that every $x\in X$ has a neighborhood $U\subset U_i$, for some orbifold chart $(\widetilde U_i, \Gamma_i, \phi_i)$, for which there exists a continuous function $v\in PSH(U)$ so that $u=v\circ\phi_i$ is smooth strictly psh on $\phi_i^{-1}(U)$, $\widetilde\Omega_i=dd^cu$ on $\phi_i^{-1}(U)$ and $\Omega=dd^cv$ on $U\cap X_{reg}$. Moreover $v$ is smooth if $x\in X_{reg}^{orb}$. To see that $v$ is strictly psh, consider any local embedding $U\hookrightarrow{\mathbb C}^N$ with coordinates $z=(z_1,\dots,z_N)$, so $\phi_i:\phi_i^{-1}(U)\longrightarrow{\mathbb C}^N$. Since $\widetilde\Omega_i$ is K\"ahler and $dd^c\|\phi_i\|^2$ is a smooth real form, by shrinking $U$ we can find $\varepsilon>0$ so that $\widetilde\Omega_i>\varepsilon dd^c\|\phi_i\|^2$ on $\phi_i^{-1}(U)$. This shows that $v\circ\phi_i-\varepsilon\|\phi_i\|^2$ is psh on $\phi_i^{-1}(U)$, hence the function $\rho(z)=v(z)-\varepsilon\|z\|^2$ is psh on $U$. Since $\rho$ extends to a psh function in the ambient space, it follows that the function $v$ is strictly psh on $U$.
\end{proof}

\medskip

\par We note that the notion of (non-closed positive) orbifold current is more restrictive than that of a (positive) current on $X_{reg}$\,. Indeed, an orbifold differential form determines a smooth form on $X_{reg}^{orb}$ whose coefficients may blow up at points of $X_{reg}\setminus X_{reg}^{orb}$\,. For instance, consider the (global) orbifold structure on ${\mathbb C}^2$ given by the analytic cover $\phi:{\mathbb C}^2\longrightarrow{\mathbb C}^2$, $\phi(z_1,z_2)=(z_1^2,z_2)$, whose orbifold singular locus is the line $\{x_1=0\}$. Then $\theta=\frac{i}{|x_1|}\,dx_1\wedge d\overline x_1$ is a smooth positive (1,1) orbifold form, since $\phi^\star\theta=4idz_1\wedge d\overline z_1$ is smooth.

\par A current $T$ on $X_{reg}$ arising from an orbifold current acts on such forms $\theta$ in the following way: assuming that ${\rm supp}\,\theta\subset U\subset X_{reg}$ for some orbifold chart $(\widetilde U,\Gamma,\phi)$, we have $T\mid_U=\frac{1}{m}\,\phi_\star\widetilde T$ for some current $\widetilde T$ on $\widetilde U$, where $m=|\Gamma|$. Hence we set $\big\langle\, T,\theta\,\big\rangle:=\frac{1}{m}\,\big\langle\,\widetilde T,\phi^\star\theta\,\big\rangle$. Returning to the previous example, we see that $T=\frac{i}{|x_1|}\,dx_2\wedge d\overline x_2$ is a positive $(1,1)$ current on ${\mathbb C}^2$ which does not arise from an orbifold current since $\int_KT\wedge\theta=+\infty$, where $K$ is the unit bidisk in ${\mathbb C}^2$.

\subsection{Singular Hermitian metrics on orbifold line bundles}\label{SS:singm}

\par We refer to \cite{D90} for the notion of singular Hermitian metric on a holomorphic line bundle over a complex manifold or complex space (see also \cite[p.\ 97]{MM07}).

\smallskip

\par Let $(L,{\mathcal X})$ be an orbifold line bundle over the orbifold ${\mathcal X}=(X,{\mathcal U})$. A singular Hermitian metric on $L$ is a collection $\widetilde h=\{\widetilde h_i\}$ of singular Hermitian metrics on the line bundles $(L_{\widetilde U_i},\widetilde U_i)$, for every orbifold chart $(\widetilde U_i, \Gamma_i, \phi_i)$, such that:

\par $(i)$ $\widetilde h_i$ is $\Gamma_i$-invariant: for every $\gamma\in\Gamma_i$ with induced linear map $\widetilde\gamma:L_{\widetilde U_i}\mid_x\longrightarrow L_{\widetilde U_i}\mid_{\gamma x}$ we have $\widetilde h_i(\widetilde\gamma\ell,\widetilde\gamma\ell')=\widetilde h_i(\ell,\ell')$ for all $\ell,\,\ell'\in L_{\widetilde U_i}\mid_x\,$, $x\in\widetilde U_i$\,.

\par $(ii)$ $\widetilde h_i$ satisfy the following gluing condition: if $\lambda_{ji}:\widetilde U_i\longrightarrow\widetilde U_j$ is an injection with associated bundle map $\widetilde\lambda_{ji}:L_{\widetilde U_j}\mid_{\lambda_{ji}(\widetilde U_i)}\longrightarrow L_{\widetilde U_i}$ then $\widetilde h_i(\widetilde\lambda_{ji}\ell,\widetilde\lambda_{ji}\ell')=\widetilde h_j(\ell,\ell')$ for all $\ell,\,\ell'\in L_{\widetilde U_j}\mid_{\lambda_{ji}(x)}$ and $x\in\widetilde U_i$.

\smallskip

\par The curvature current
$$c_1(L,\widetilde h):=\{c_1(L_{\widetilde U_i},\widetilde h_i)\}$$
is a well defined real closed (1,1) orbifold current. Indeed:

\par $(i)$ $c_1(L_{\widetilde U_i},\widetilde h_i)$ is $\Gamma_i$-invariant: by shrinking $U_i$ we may assume that $L_{\widetilde U_i}$ has a holomorphic frame $\widetilde e_i$ on $\widetilde U_i$, so $\widetilde h_i(\widetilde e_i,\widetilde e_i)=e^{-2\psi_i}$ for some function $\psi_i\in L^1_{loc}(\widetilde U_i)$. As $\widetilde\gamma\widetilde e_i$ is also a frame on $\widetilde U_i$, we have $\widetilde\gamma\widetilde e_i(x)=f(x)\widetilde e_i(\gamma x)$ for some non-vanishing holomorphic function $f$ on $\widetilde U_i$ and $|f|^2e^{-2\psi_i\circ\gamma}=\widetilde h_i(\widetilde\gamma\widetilde e_i,\widetilde\gamma\widetilde e_i)=\widetilde h_i(\widetilde e_i,\widetilde e_i)=e^{-2\psi_i}$. Hence $dd^c\psi_i\circ\gamma=dd^c\psi_i$, so $\gamma^\star c_1(L_{\widetilde U_i},\widetilde h_i)=c_1(L_{\widetilde U_i},\widetilde h_i)$.

\par $(ii)$ If $\lambda_{ji}:\widetilde U_i\longrightarrow\widetilde U_j$ is an injection then $\lambda_{ji}^\star(c_1(L_{\widetilde U_j},\widetilde h_j)\mid_{\lambda_{ji}(\widetilde U_i)})=c_1(L_{\widetilde U_i},\widetilde h_i)$: if $\widetilde e_j$ is a frame of $L_{\widetilde U_j}$ over some set $\lambda_{ji}(V)$ where $V\subset\widetilde U_i$ is open, then $\widetilde\lambda_{ji}\widetilde e_j$ is a frame of $L_{\widetilde U_i}$ over $V$, so $\widetilde\lambda_{ji}\widetilde e_j(\lambda_{ji}(x))=f(x)\widetilde e_i(x)$ for some non-vanishing holomorphic function $f$ on $V$. Thus
$$|f|^2e^{-2\psi_i}=\widetilde h_i(\widetilde\lambda_{ji}\widetilde e_j\lambda_{ji},\widetilde\lambda_{ji}\widetilde e_j\lambda_{ji})=\widetilde h_j(\widetilde e_j\lambda_{ji},\widetilde e_j\lambda_{ji})=e^{-2\psi_j\circ\lambda_{ji}}\;\;{\rm on}\;V\,,$$
which shows that $\lambda_{ji}^\star(dd^c\psi_j)=dd^c\psi_i$ on $V$.

\smallskip

\par We say that the metric $\widetilde h$ is (semi)positively curved if its curvature $c_1(L,\widetilde h)$ is a positive current.

\smallskip

\begin{Lemma}\label{L:secnorm} Let $(L,{\mathcal X})$ be an orbifold line bundle over the orbifold ${\mathcal X}=(X,{\mathcal U})$ endowed with a singular Hermitian metric $\widetilde h=\{\widetilde h_i\}$ and let $S=\{\widetilde S_i\}$ be a section of $L$. There exists a function on $X$, denoted by $|S|^2_{\widetilde h}$, so that for every orbifold chart $(\widetilde U_i, \Gamma_i, \phi_i)$ we have
$|S|^2_{\widetilde h}\circ\phi_i=|\widetilde S_i|^2_{\widetilde h_i}:=\widetilde h_i(\widetilde S_i,\widetilde S_i)$.
\end{Lemma}

\begin{proof} Note that the function $|\widetilde S_i|^2_{\widetilde h_i}$ is $\Gamma_i$-invariant, as
$$\widetilde h_i\big(\widetilde S_i(\gamma(x)),\widetilde S_i(\gamma(x))\big)=\widetilde h_i\big(\widetilde\gamma\widetilde S_i(x),\widetilde\gamma\widetilde S_i(x)\big)=\widetilde h_i\big(\widetilde S_i(x),\widetilde S_i(x)\big)\,,\;\forall\,\gamma\in\Gamma_i\,.$$
Hence there exists as function $f_i$ on $U_i$ so that $|\widetilde S_i|^2_{\widetilde h_i}=f_i\circ\phi_i$. We have to show that if $\lambda_{ji}:\widetilde U_i\longrightarrow\widetilde U_j$ is an injection then $f_j\mid_{U_i}=f_i$. Indeed,
\begin{eqnarray*}
f_j\circ\phi_i(x)&=&f_j\circ\phi_j\circ\lambda_{ji}(x)=\widetilde h_j\big(\widetilde S_j(\lambda_{ji}(x)),\widetilde S_j(\lambda_{ji}(x))\big)\\
&=&\widetilde h_i\big(\widetilde\lambda_{ji}\widetilde S_j(\lambda_{ji}(x)),\widetilde\lambda_{ji}\widetilde S_j(\lambda_{ji}(x))\big)=\widetilde h_i\big(\widetilde S_i(x),\widetilde S_i(x)\big)=f_i\circ\phi_i(x),
\end{eqnarray*}
for every $x\in\widetilde U_i$.
\end{proof}

\begin{Remark}\label{R:secnorm} If $x\in X_{reg}^{orb}$ there exists an orbifold chart $\phi_i:\widetilde U_i\longrightarrow U_i$ so that $x\in U_i$ and $\phi_i$ is biholomorphic. Thus $L\mid_{U_i}\cong(\phi_i^{-1})^\star L_{\widetilde U_i}$ is a holomorphic line bundle with a singular Hermitian metric $h_i$ induced by $\widetilde h_i$. It follows that the holomorphic line bundle $L\mid_{X_{reg}^{orb}}$ has a singular Hermitian metric $h$ induced by $\widetilde h$ and
$$|S|_h^2:=h(S,S)=|S|^2_{\widetilde h}\mid_{X_{reg}^{orb}}$$
for every orbifold section $S:X\longrightarrow L$. Moreover, the curvature current
$$c_1(L\mid_{X_{reg}^{orb}},h)=\boldsymbol{F}(c_1(L,\widetilde h))\mid_{X_{reg}^{orb}}\,,$$
where $\boldsymbol{F}$ is the map constructed in Proposition \ref{P:pqc}.
\end{Remark}

\subsection{Bergman kernel and Fubini-Study currents}\label{SS:FS}

\par Let $(L,\mathcal X)$ be an orbifold line bundle over ${\mathcal X}=(X,\mathcal U)$ endowed with a singular Hermitian metric $\widetilde h=\{\widetilde h_i\}$, and let $\widetilde\Omega=\{\widetilde \Omega_i\}$ be a Hermitian form on $\mathcal X$. Here ${\mathcal U}=\{(\widetilde U_i,\Gamma_i,\phi_i)\}$, $U_i=\phi_i(\widetilde U_i)$, $m_i=|\Gamma_i|$. We denote by $h$ the singular Hermitian metric induced by $\widetilde h$ on the holomorphic line bundle $L\mid_{X_{reg}^{orb}}$\,. By Proposition \ref{P:pqc}, $\widetilde\Omega$ induces a positive (1,1) current $\Omega$ on $X_{reg}$, which clearly is a smooth Hermitian form on $X_{reg}^{orb}$ with $\phi_i^\star(\Omega\mid_{U_i})=\widetilde\Omega_i$.

\par The space $H_{(2)}^0(\mathcal X,L)$ of $L^2$-holomorphic sections with respect to this metric data is defined as follows. Fix a partition of unity $\{\chi_\ell\}_{l\geq1}$ on $X$ so that $\chi_\ell$ has compact support contained in $U_{i_\ell}$ and set
$$\|S\|^2=\sum_{l=1}^\infty\frac{1}{m_{i_\ell}}\,\int_{\widetilde U_{i_\ell}}(\chi_\ell\circ\phi_{i_\ell})|\widetilde S_{i_\ell}|_{\widetilde h_{i_\ell}}^2\,\widetilde\Omega_{i_\ell}^n\,,\;\;{\rm where}\;S=\{\widetilde S_i\}\in H^0(\mathcal X,L)\,.$$
Define $$H_{(2)}^0(\mathcal X,L)=\{S\in H^0(\mathcal X,L):\,\|S\|^2<+\infty\}\,,$$
endowed with the obvious inner product. If $A_i=U_i\cap X_{sing}^{orb}$ and $\widetilde A_i=\phi_i^{-1}(A_i)$, we note that
$$\frac{1}{m_{i_\ell}}\,\int_{\widetilde U_{i_\ell}\setminus\widetilde A_{i_\ell}}(\chi_\ell\circ\phi_{i_\ell})|\widetilde S_{i_\ell}|_{\widetilde h_{i_\ell}}^2\,\widetilde\Omega_{i_\ell}^n=\int_{ U_{i_\ell}\setminus A_{i_\ell}}\chi_\ell\,|S|_{\widetilde h}^2\;\Omega^n\,.$$
It follows that
$$\|S\|^2=\int_{X_{reg}^{orb}}|S|_{\widetilde h}^2\,\Omega^n=\int_{X_{reg}^{orb}}|S|_h^2\,\Omega^n\,,$$
where $|S|_{\widetilde h}^2$ is the function from Lemma \ref{L:secnorm}, and the same holds for the inner product.

\medskip

\par Since $H^0_{(2)}(\mathcal X,L)$ is separable, let $\{S_j\}_{j\geq1}$ be an orthonormal basis and denote by $P$ the function defined on $X$ by
\begin{equation}\label{e:Bergfcn}
P=\sum_{j=1}^\infty|S_j|_{\widetilde h}^2\,.
\end{equation}
This function is independent of the choice of basis (see Lemma \ref{L:Bergfcn}) and it is called the {\em Bergman kernel function} associated to the space $H^0_{(2)}(\mathcal X,L)$.

\par The \emph{orbifold Fubini-Study current} $\widetilde\alpha=\{\widetilde\alpha_i\}$ is defined as follows: given a chart $(\widetilde U_i,\Gamma_i,\phi_i)$ we may assume that $L_{\widetilde U_i}$ has a holomorphic frame $\widetilde e_i$ on $\widetilde U_i$. If $S_j=\{\widetilde S_{j,i}\}$ we write $\widetilde S_{j,i}=s_{j,i}\widetilde e_i$ for some holomorphic functions $s_{j,i}$ on $\widetilde U_i$, and we set
\begin{equation}\label{e:FSc}
\widetilde\alpha_i=\frac{1}{2}\,dd^c\log\left(\sum_{j=1}^\infty|s_{j,i}|^2\right)\,,
\end{equation}
where $d^c=\frac{1}{2\pi i}(\partial-\overline\partial)$. To explain the terminology, let us assume
that $H^0_{(2)}(\mathcal{X},L)$ is finite dimensional and non-trivial. Denote by $\omega_{FS}$ the Fubini-Study metric on $\mathbb{P}(H^0_{(2)}(\mathcal{X},L)^*)$ induced by the $L^2$ inner product on $H^0_{(2)}(\mathcal{X},L)$. 
Consider the Kodaira map 
\[\Phi:X\dashrightarrow\mathbb{P}(H^0_{(2)}(\mathcal{X},L)^*)\,,\quad x\mapsto\big\{S\in H^0_{(2)}(\mathcal{X},L):S(x)=0\big\}.\]
Then $\widetilde\alpha=\{\widetilde\alpha_i\}$, $\widetilde\alpha_i=(\Phi\circ\phi_i)^*\omega_{FS}$\,, is the orbifold Fubini-Study current.

\medskip

\par For the convenience of the reader, we include a proof of some properties of these notions in our setting.

\begin{Lemma}\label{L:Bergfcn} If the singular metrics $\widetilde h_i$ have locally upper bounded weights, then:

\par (i) $\widetilde\alpha$ is a well defined positive closed current of bidegree (1,1) on $\mathcal X$.

\par  (ii) The function $P$ is independent of the choice of basis $\{S_j\}_{j\geq1}$ and
\begin{equation}\label{e:mBK}
P(x)=\max\big\{|S|^2_{\widetilde h}(x):\,S\in H^0_{(2)}(\mathcal X,L),\;\|S\|=1\big\}\,,\;\;\text{for all $x\in X$}\,.
\end{equation}

\par $(iii)$ On each chart, $\log P\circ\phi_i\in L^1_{loc}(\widetilde U_i,\widetilde\Omega_i^n)$ and
$$2\widetilde\alpha_i=2c_1(L\mid_{\widetilde U_i},\widetilde h_i)+dd^c\log P\circ\phi_i\,,$$
so $\widetilde\alpha$  is independent of the choice of basis $\{S_j\}_{j\geq1}$.
\end{Lemma}

\begin{proof} By the Riesz-Fischer theorem we have that $S\in H^0_{(2)}(\mathcal X,L)$ if and only if there exists a sequence $a=\{a_j\}\in l^2$ so that $S=S_a$, where $S_a=\sum_{j=1}^\infty a_jS_j$ and  $\|S_a\|=\|a\|_2$\,.

\smallskip

\par Given a chart $U_i\subset X$ so that $L_{\widetilde U_i}$ has a holomorphic frame $\widetilde e_i$ on $\widetilde U_i$, we write $S_j=\{\widetilde S_{j,\,i}\}$,  $S_a=\{\widetilde S_{a,\,i}\}$, $\widetilde S_{j,\,i}=s_{j,\,i}\widetilde e_i$, $\widetilde S_{a,\,i}=s_{a,\,i}\widetilde e_i$, with holomorphic functions $s_{j,\,i},\,s_{a,\,i}$ on $\widetilde U_i$, and $|\widetilde e_i|^2_{\widetilde h_i}=e^{-2\psi_i}$. It follows that $s_{a,\,i}=\sum_{j=1}^\infty a_js_{j,\,i}$ and the series converges locally uniformly on $\widetilde U_i$. Indeed, if $K_1\Subset K_2\subset\widetilde U_i$ are compact sets then, since $\psi_i$ is locally upper bounded, we have
\begin{eqnarray*}
\max_{K_1}\big|\sum_{j=N}^Ma_js_{j,\,i}\big|^2 & \leq & C_1\int_{K_2}\big|\sum_{j=N}^Ma_js_{j,\,i}\big|^2\,\widetilde\Omega_i^n\leq C_2\int_{K_2}\big|\sum_{j=N}^Ma_j\widetilde S_{j,\,i}\big|_{\widetilde h_i}^2\,\widetilde\Omega_i^n \\
& \leq & C_2m_i\int_{\phi_i(K_2)\cap X_{reg}^{orb}}\big|\sum_{j=N}^Ma_jS_j\big|_{\widetilde h}^2\,\Omega^n \\
& \leq & C_2m_i\big\|\sum_{j=N}^Ma_jS_j\big\|^2=C_2m_i\sum_{j=N}^M|a_j|^2\,.
\end{eqnarray*}
As this holds for every sequence $a\in l^2$ we see that $\{s_{j,\,i}(z)\}_{j\geq1}\in l^2$ for all $z\in\widetilde U_i$. Using this and the same argument from the proof of \cite[Lemma 3.1]{CM11} we show that the series $\sum_{j=1}^\infty|s_{j,\,i}|^2$ converges locally uniformly on $\widetilde U_i$\,, so its logarithm is a psh function and the current $\widetilde\alpha_i$ is a positive closed current on $\widetilde U_i$. The proof that $\widetilde\alpha\in\widetilde{\mathcal T}$ is similar to the one showing that $c_1(L,\widetilde h)$ is a well-defined orbifold current.

\smallskip

\par By Lemma \ref{L:secnorm} we see that
$$P\circ\phi_i=\sum_{j=1}^\infty|\widetilde S_{j,\,i}|^2_{\widetilde h_i}=e^{-2\psi_i}\sum_{j=1}^\infty|s_{j,\,i}|^2,$$
which implies $(iii)$. To prove $(ii)$ we let $x\in U_i$, for $U_i$ as above. If $a\in l^2$, $\|a\|_2=1$, and $S_a=\sum_{j=1}^\infty a_jS_j$ then on $\widetilde U_i$, by the Cauchy-Schwarz inequality,
$$|S_a|^2_{\widetilde h}\circ\phi_i=|\widetilde S_{a,\,i}|^2_{\widetilde h_i}\leq\sum_{j=1}^\infty|\widetilde S_{j,\,i}|^2_{\widetilde h_i}=P\circ\phi_i\,.$$
Moreover, if $y\in\phi_i^{-1}(x)$ set
$$a=\left\{c^{-1}\,\overline{s_{j,\,i}(y)}\right\}_{j\geq1},\;c:=\left(\sum_{j=1}^\infty|s_{j,\,i}(y)|^2\right)^{1/2}.$$
Then $\|a\|_2=1$, $\widetilde S_{a,\,i}(y)=c\widetilde e_i$, so $|S_a|^2_{\widetilde h}(x)=|\widetilde S_{a,\,i}(y)|^2_{\widetilde h_i}=|c|^2e^{-2\psi_i(y)}=P(x)$.
\end{proof}

\section{Proofs of Theorems \ref{T:mt1}, \ref{T:mt2} and \ref{T:mt3}}\label{S:pfmt}
Now we will prove the main results. In Section \ref{SS:extsec} we examine the extension of holomorphic square integrable sections defined on the orbifold regular locus. This yields the fact that the logarithm of the Bergman kernel is locally the difference of two psh functions.
In Section \ref{SS:dbar} we recall the $L^2$ estimates for $\overline\partial$ in the form we use them. In Section \ref{SS:pfmt} we prove the weak asymptotics of the Bergman kernel, $\frac{1}{p}\log P_p\rightarrow 0$ as $p\to\infty$, in $L^1_{loc}$, and deduce Theorems \ref{T:mt1}, \ref{T:mt2} and \ref{T:mt3}.

\subsection{Extension of holomorphic sections}\label{SS:extsec}

\par We begin with two lemmas, which are formulated in a general context. Let $(L,\mathcal X)$ be an orbifold line bundle over ${\mathcal X}=(X,\mathcal U)$ endowed with a singular metric $\widetilde h=\{\widetilde h_i\}$, and let $\widetilde\Omega=\{\widetilde \Omega_i\}$ be a Hermitian form on $\mathcal X$. As before, we set ${\mathcal U}=\{(\widetilde U_i,\Gamma_i,\phi_i)\}$, $U_i=\phi_i(\widetilde U_i)$, $m_i=|\Gamma_i|$. Let $h$ be the singular metric induced by $\widetilde h$ on the line bundle $L\mid_{X_{reg}^{orb}}\,$, and $\Omega$ be the Hermitian form induced by $\widetilde\Omega$ on $X_{reg}^{orb}$\,.

\begin{Lemma}\label{L:Sext} Assume that the singular metrics $\widetilde h_i$ have weights that are locally upper bounded and that
$$S\in H^0(X_{reg}^{orb},L\mid_{X_{reg}^{orb}})\,,\;\;\|S\|^2=\int_{X_{reg}^{orb}}|S|_h^2\,\Omega^n<+\infty\,.$$
Then $S$ extends to a holomorphic section of $L$ over $\mathcal X$ and $S\in H_{(2)}^0(\mathcal X,L)$.
\end{Lemma}

\begin{proof} Without loss of generality, we can consider an orbifold chart so that $L_{\widetilde U_i}$ has a holomorphic frame $\widetilde e_i$ on $\widetilde U_i$ and we let $|\widetilde e_i|^2_{\widetilde h_i}=e^{-2\psi_i}$. Set $A_i=U_i\cap X_{sing}^{orb}$ and $\widetilde A_i=\phi_i^{-1}(A_i)$. The action of $\gamma\in\Gamma_i$ on $L_{\widetilde U_i}$ is defined by $\widetilde\gamma(\widetilde e_i(y))=h_i(\gamma)\widetilde e_i(\gamma y)$, where $h_i$ is a group homomorphism of $\Gamma_i$ to the group of roots of order $m_i$ of unity.

\par Using the notation from Section \ref{SS:olb}, it follows that if $y\in\widetilde U_i\setminus\widetilde A_i$ we can define $\widetilde S_i(y)$ as the unique element of the set $\widetilde\phi_i^{-1}(S(\phi_i(y)))$ that lies in the fiber $L_{\widetilde U_i}\mid_y$, and $\widetilde S_i$ is a holomorphic section of $L_{\widetilde U_i}$ over $\widetilde U_i\setminus\widetilde A_i$. Writing $\widetilde S_i=f\widetilde e_i$, we see that
$$\widetilde\phi_i^{-1}(S(\phi_i(y)))=\{f(y)h_i(\gamma)\widetilde e_i(\gamma y):\,\gamma\in\Gamma_i\}\,,$$
so $\widetilde\gamma(\widetilde S_i(y))=f(y)h_i(\gamma)\widetilde e_i(\gamma y)=\widetilde S_i(\gamma(y))$. Since $\psi_i$ is locally upper bounded we have for any compact $K\subset\widetilde U_i$,
$$\int_{K\setminus\widetilde A_i}|f|^2\,\widetilde\Omega_i^n\leq C_K\int_{K\setminus\widetilde A_i}|\widetilde S_i|^2_{\widetilde h_i}\,\widetilde\Omega_i^n\leq C_Km_i\int_{\phi_i(K)\setminus A_i}|S|_h^2\,\Omega^n<+\infty\,.$$
By Skoda's lemma \cite[Lemma\,2.3.22]{MM07}, we conclude that $\widetilde S_i$ extends to an equivariant holomorphic section of $L_{\widetilde U_i}$ over $\widetilde U_i$. If $\lambda_{ji}:\widetilde U_i\longrightarrow\widetilde U_j$ is an injection we have $\lambda_{ji}(\widetilde U_i\setminus\widetilde A_i)=\lambda_{ji}(\widetilde U_i)\setminus\widetilde A_j$. Indeed, for $y\in\widetilde U_i$ we have that the stabilizer of $\lambda_{ji}(y)$ is $(\Gamma_j)_{\lambda_{ji}(y)}=\lambda_{ji}((\Gamma_i)_y)$ (see \eqref{e:coset}). Using this we verify that the gluing condition $\widetilde\lambda_{ji}\circ\widetilde S_j\circ\lambda_{ji}=\widetilde S_i$ holds on $\widetilde U_i\setminus\widetilde A_i$, hence on $\widetilde U_i$.
\end{proof}

\begin{Lemma}\label{L:symp} Assume that the current $c_1(L,\widetilde h)$ is positive and let $\widetilde\alpha=\{\widetilde\alpha_i\}$ be the Fubini-Study current associated to $H_{(2)}^0(\mathcal X,L)$. Then every point $x\in X$ has a neighborhood $U$ contained in some chart $U_i$ on which there exist psh functions $v=v_x(\widetilde\alpha)$, $u=u_x(c_1(L,\widetilde h))$, so that $\widetilde\alpha_i=dd^c(v\circ\phi_i)$, $c_1(L\mid_{\widetilde U_i},\widetilde h_i)=dd^c(u\circ\phi_i)$ hold on $\phi_i^{-1}(U)$, and $2v-2u=\log P$ holds on $U$.
\end{Lemma}

\begin{Corollary}\label{C:Pdpsh} The function $\log P$ is locally the difference of two psh functions.
\end{Corollary}
\begin{proof}[Proof of Lemma \ref{L:symp}]
Let $x\in U_i$. Assuming that $\widetilde e_i$ is a frame for $  L_{\widetilde U_i}$ we write $|\widetilde e_i|^2_{\widetilde h_i}=e^{-2\psi_i}$, where $\psi_i\in PSH(\widetilde U_i)$. By \eqref{e:FSc} and Lemma \ref{L:Bergfcn} we have
$$\widetilde\alpha_i=dd^c v_i\,,\;\;v_i=\frac{1}{2}\,\log\left(\sum_{j=1}^\infty|s_{j,i}|^2\right)\in PSH(\widetilde U_i)\,,\;2v_i-2\psi_i=\log P\circ\phi_i\,.$$
Letting $\widetilde v=m_i^{-1}\sum_{\gamma\in\Gamma_i}v_i\circ\gamma$, $\psi=m_i^{-1}\sum_{\gamma\in\Gamma_i}\psi_i\circ\gamma$, we obtain $\Gamma_i$-invariant psh functions so that $\widetilde\alpha_i=dd^c\widetilde v$, $c_1(L\mid_{\widetilde U_i},\widetilde h_i)=dd^c\psi$. Hence $\widetilde v=v\circ\phi_i$, $\psi=u\circ\phi_i$, for psh functions $v\,,u$ on $U_i$. Since $\phi_i\circ\gamma=\phi_i$ we have $2v_i\circ\gamma-2\psi_i\circ\gamma=\log P\circ\phi_i$, so $2\widetilde v-2\psi=\log P\circ\phi_i$. 
\end{proof}

\medskip

\subsection{Demailly's estimates for $\overline\partial$}\label{SS:dbar}

\par In order to prove our theorems we need the following variants of the existence theorem for $\overline\partial$ in the case of singular Hermitian metrics due to Demailly \cite{D82}. Let $(M,\Omega)$ be a Hermitian manifold of dimension $n$, $(L,h)$ be a singular Hermitian holomorphic line bundle. Let $\theta\geq0$ be a $(1,1)$-current and let $\theta_{abs}$ be the absolute continuous component in the Lebesgue decomposition of $\theta$. For any $\alpha\in\Lambda^{n,1}T_x^*X\otimes L_x$, $x\in M$, we define $|\alpha|_{\Omega,\,\theta}\in[0,\infty]$ to be the smallest number
satisfying
\[
\big|\big\langle\,\alpha,\beta\,\big\rangle\big|^2\leq\big|\alpha\big|_{\Omega,\,\theta}^2\,\big\langle\,\theta_{abs}\wedge\Lambda_\Omega\beta,\beta\,\big\rangle\,,\quad\text{for all $\beta\in\Lambda^{n,1}T_x^*X\otimes L_x$},
\]
where $\Lambda_\Omega$ is the interior product with $\Omega$. This number is independent of the metric: if $\Omega_1$ is another metric on $M$, then
$|\alpha|^2_{\Omega_1,\,\theta}\,\Omega^n_1=|\alpha|^2_{\Omega,\,\theta}\,\Omega^n$, cf.\ \cite[Lemme\,3.2\,(3.2)]{D82}. Note also that for any $\alpha\in\Lambda^{n,0}T_x^*X\otimes L_x$, $x\in M$, we have $|\alpha|^2_{\Omega_1}\,\Omega^n_1=|\alpha|^2_{\Omega}\,\Omega^n$, cf.\ \cite[Remarque 4.5]{D82}. Hence for $(n,1)$-forms the estimate in the following theorem is independent of the K\"ahler metric.

\begin{Theorem}[{\cite[Th\'eor\`eme 5.1]{D82}}]\label{l2est}
Let $(M,\Omega)$ be a K\"ahler manifold of dimension $n$ which admits a complete K\"ahler metric. Let $(L,h)$ a  singular Hermitian holomorphic line bundle such that $\theta=c_1(L,h)\geq0$.
Then for any form $g\in L_{n,1}^2(M,L,loc)$ satisfying
\[
{\overline\partial}g=0\,,\quad \int_M|g|^2_{\Omega,\,\theta}\,\Omega^n<+\infty
\]
 there exists $u\in L_{n,0}^2(M,L,loc)$ with $\overline\partial u=g$ and
$$\int_M|u|^2_{\Omega}\,\Omega^n\leq\int_M|g|^2_{\Omega,\,\theta}\,\Omega^n\,.$$
\end{Theorem}

\begin{Corollary}\label{l2est2}
Let $(M,\Omega)$ be a K\"ahler manifold of dimension $n$ which admits a complete K\"ahler metric. Let $(L,h)$ a  singular Hermitian holomorphic line bundle and let $\lambda:M\to[0,+\infty)$ be a continuous function such that $c_1(L,h)\geq\lambda\Omega$.
Then for any form $g\in L_{n,1}^2(M,L,loc)$ satisfying
\[
{\overline\partial}g=0\,,\quad \int_M\lambda^{-1}|g|^2\,\Omega^n<+\infty
\]
 there exists $u\in L_{n,0}^2(M,L,loc)$ with $\overline\partial u=g$ and
$$\int_M|u|^2\,\Omega^n\leq\int_M\lambda^{-1}|g|^2\,\Omega^n\,.$$
\end{Corollary}

\begin{proof} We have $c_1(L,h)_{abs}\geq\lambda\Omega$, so by \cite[Remarque 4.2]{D82}, $|\alpha|_{\Omega,\,\theta}^2\leq\lambda(x)^{-1}|\alpha|^2$, for any $\alpha\in\Lambda^{n,1}T_x^*X\otimes L_x$ (with the conventions $\frac10=+\infty$, $0\cdot\infty=0$). Hence the conclusion follows immediately from Theorem \ref{l2est}.
\end{proof}

\begin{Corollary}\label{l2est23}
Let $(M,\Omega)$ be a K\"ahler manifold such that $Ric_{\Omega}\geq0$. Assume that $M$ carries a complete K\"ahler metric. Let $(L,h)$ a  singular Hermitian holomorphic line bundle and let $\lambda:M\to[0,+\infty)$ be a continuous function such that $c_1(L,h)\geq\lambda\Omega$. Then for any form $g\in L_{0,1}^2(M,L,loc)$ satisfying
\[
{\overline\partial}g=0\,,\quad \int_M\lambda^{-1}|g|^2\,\Omega^n<+\infty
\]
 there exists $u\in L_{0,0}^2(M,L,loc)$ with $\overline\partial u=g$ and
$$\int_M|u|^2\,\Omega^n\leq\int_M\lambda^{-1}|g|^2\,\Omega^n\,.$$
\end{Corollary}

\begin{proof} We apply Corollary \ref{l2est2} to the line bundle $(F,h^F)=(L\otimes K_X^{*},h\otimes h^{K_X^{*}})$, where $h^{K_X^{*}}$ is the metric induced by $\Omega$. Obviously, $c_1(F,h^F)\geq\lambda\Omega$. There exists a natural isometry
\begin{equation*}
\begin{split}
&\Psi=\thicksim\,:\Lambda^{0,q}(T^*X)\otimes L
\longrightarrow\Lambda^{n,q}(T^*X)\otimes F,\\
&\Psi \, s=\widetilde s=(w^1\wedge\ldots\wedge w^n\wedge s)\otimes
(w_1\wedge\ldots\wedge w_n),
\end{split}
\end{equation*}
where $\{w_j\}^n_{j=1}$ is a local holomorphic frame of $T^{(1,0)}X$ and $\{w^j\}^n_{j=1}$ is the dual frame. The operator $\Psi$ commutes with the action of $\overline\partial$. For a form $g\in L_{0,1}^2(M,L,loc)$ with $\overline\partial g=0$, $\int_M\lambda^{-1}|g|^2\,\Omega^n<+\infty$ we have $\widetilde{g}\in L_{n,1}^2(M,L,loc)$, $\overline\partial\widetilde{g}=0$ and $\int_M\lambda^{-1}|\widetilde{g}|^2\,\Omega^n<+\infty$. By Corollary \ref{l2est2}, there exists $\widetilde{u}\in L_{n,0}^2(M,L,loc)$ with $\overline\partial\widetilde{u}=\widetilde{g}$ and
$\int_M|\widetilde{u}|^2\,\Omega^n\leq\int_M\lambda^{-1}|\widetilde{g}|^2\,\Omega^n$. Then $u=\Psi^{-1}\widetilde{u}$ satisfies the conclusion.
\end{proof}

\medskip

\subsection{Proofs of Theorems \ref{T:mt1}, \ref{T:mt2}, \ref{T:mt3}}\label{SS:pfmt}

\par Theorem \ref{T:mt1} will follow from Lemmas \ref{L:Sext}, \ref{L:symp}, and from:

\begin{Theorem}\label{T:Pp} In the setting of Theorem \ref{T:mt1}, we  have that $\frac{1}{p}\,\log P_p\to0$ as $p\to\infty$, in $L_{loc}^1(G\cap X_{reg}^{orb},\Omega^n)$.
\end{Theorem}

\begin{proof} Let $X'=X_{reg}^{orb}$ and recall that $L\mid_{X'}$ is a holomorphic line bundle with a metric $h$ induced by $\widetilde h$. We denote by $h_p$ the metric induced by $h$ and $h^{K_{X'}}$ on $(L\mid_{X'})^p\otimes K_{X'}$. By Lemma \ref{L:Sext}, $P_p$ is the Bergman kernel function of the space $H_{(2)}^0(X',(L\mid_{X'})^p\otimes K_{X'})=H_{(2)}^0({\mathcal X},L^p\otimes K_{\mathcal X})$, where the norm is denoted by $\|\cdot\|$.

\par We let $x\in G\cap X'$ and $U_\alpha\subset G\cap X'$ be a coordinate neighborhood of $x$ on which there exists a holomorphic frame $e_\alpha$ of $L\mid_{X'}$ and $e'_\alpha$ of $K_{X'}$. Let $\psi_\alpha$ be a psh weight of $h$ and $\rho_\alpha$ be a smooth weight of $h^{K_{X'}}$ on $U_\alpha$. Fix $r_0>0$ so that the ball $V:=B(x,2r_0)\Subset U_\alpha$ and let $U:=B(x,r_0)$.

\par Following the arguments of \cite{D92,CM11} we will show that there exist constants $C>0$, $p_0\in\mathbb{N}$ so that
\begin{equation}\label{e:Bke}
-\frac{\log C}{p}\leq\frac{1}{p}\,\log P_p(z)\leq\frac{\log(Cr^{-2n})}{p}+2\left(\max_{B(z,r)}\psi_\alpha-\psi_\alpha(z)\right)
\end{equation}
holds for all $p>p_0$, $0<r<r_0$ and $z\in U$ with $\psi_\alpha(z)>-\infty$. By \eqref{e:Bke} it follows that $\frac{1}{p}\,\log P_p\to0$ in $L^1(U,\Omega^n)$, as in \cite[Theorem 5.1]{CM11}.

\smallskip

\par For the upper estimate, fix $z\in U$ with $\psi_\alpha(z)>-\infty$ and $r<r_0$. Let $S\in H_{(2)}^0(X',(L\mid_{X'})^p\otimes K_{X'})$ with $\|S\|=1$ and write $S=s\,e_\alpha^{\otimes p}\otimes e'_\alpha$. Then
\begin{eqnarray*}
|S(z)|^2_{h_p}&=&|s(z)|^2e^{-2p\psi_\alpha(z)-2\rho_\alpha(z)}\leq e^{-2p\psi_\alpha(z)}\frac{C_1}{r^{2n}}\,\int_{B(z,r)}|s|^2\,\Omega^n\\
&\leq&\frac{C_2}{r^{2n}}\,\exp\left(2p\left(\max_{B(z,r)}\psi_\alpha-\psi_\alpha(z)\right)\right)\int_{B(z,r)}|s|^2e^{-2p\psi_\alpha-2\rho_\alpha}\,\Omega^n\\
&\leq&\frac{C_2}{r^{2n}}\,\exp\left(2p\left(\max_{B(z,r)}\psi_\alpha-\psi_\alpha(z)\right)\right),
\end{eqnarray*}
where $C_2$ is a constant that depends only on $x$. Hence by Lemma \ref{L:Bergfcn}
$$\frac{1}{p}\,\log P_p(z)=\frac{1}{p}\,\max_{\|S\|=1}\log |S(z)|^2_{h_p}\leq\frac{\log(C_2r^{-2n})}{p}+2\left(\max_{B(z,r)}\psi_\alpha-\psi_\alpha(z)\right).$$

\smallskip

\par We prove next the lower estimate from \eqref{e:Bke}. We proceed like in the proof of \cite[Theorem 5.1]{CM11} by using an argument of \cite[Section 9]{D93b} to show that there exist a constant $C_1>0$ and $p_0\in\mathbb{N}$ such that for all $p>p_0$ and all $z\in U$ with $\psi_\alpha(z)>-\infty$ there is a section $S_{z,p}\in H_{(2)}^0(X',(L\mid_{X'})^p\otimes K_{X'})=H_{(2)}^0({\mathcal X},L^p\otimes K_{\mathcal X})$ with $S_{z,p}(z)\neq0$ and
\begin{equation}\label{e:sp}
\|S_{z,p}\|^2\leq C_1|S_{z,p}(z)|^2_{h_p}\,.
\end{equation}
Observe that \eqref{e:mBK} and \eqref{e:sp} yield the desired lower estimate
\[
\frac{1}{p}\,\log P_p(z)=\frac{1}{p}\,\max_{\|S\|=1}\log|S(z)|^2_{h_p}\geq-\frac{\log C_1}{p}\,\cdot
\]
\par Let us prove the existence of $S_{z,p}$ as above. By the Ohsawa-Takegoshi extension theorem \cite{OT87} there exists a constant $C'>0$ (depending only on $x$) such that for any $z\in U$ and any $p$ there exists a function $v_{z,p}\in{\mathcal O}(V)$ with $v_{z,p}(z)\neq0$ and
\begin{equation}\label{e:ohst}
\int_V|v_{z,p}|^2e^{-2p\psi_\alpha}\Omega^n\leq C'|v_{z,p}(z)|^2e^{-2p\psi_\alpha(z)}\,.
\end{equation}

\par We shall now solve the $\overline\partial$-equation with $L^2$-estimates in order to extend $v_{z,p}$ to a section of $(L\mid_{X'})^p\otimes K_{X'}$ over $X'$. Let $\theta\in\mathcal C^\infty(\mathbb R)$ be a cut-off function such that $0\leq\theta\leq1$, $\theta(t)=1$ for $|t|\leq\frac12$, $\theta(t)=0$ for $|t|\geq1$.
Define the quasi-psh function $\varphi_z$ on $X'$ by
\begin{equation}\label{e:weight}
\varphi_z(y)=\begin{cases}2n\theta\big(\tfrac{|y-z|}{r_0}\big)\log\frac{|y-z|}{r_0}\,,\quad\text{for $y\in U_\alpha$}\,,\\
0,\quad\text{for $y\in X'\setminus B(z,r_0)$}\,.
\end{cases}
\end{equation}
Then there exists $C>0$ such that $dd^c\varphi_z\geq -C\Omega$ on $X'$ for all $z\in U$.
Since the function $\varepsilon>0$ is continuous, there exists a constant $a'>0$ such that \[c_1(L\mid_{X'},h)\geq a'\Omega \quad\text{on a neighborhoof of $\overline{V}$}.\] Therefore there exist $a>0$, $p_0\in\mathbb{N}$ such that for all $p\geq p_0$ and all $z\in U$
\[
\begin{split}
&c_1((L\mid_{X'})^p,h^pe^{-\varphi_z})\geq0\;\;\text{on $X'$}\,,\\ &c_1((L\mid_{X'})^p,h^pe^{-\varphi_z})\geq ap\,\Omega\;\;\text{on a neighborhoof of $\overline{V}$}.
\end{split}
\]
Let $\lambda:X'\to[0,+\infty)$ be a continuous function such that $\lambda=ap$ on $\overline{V}$ and  \[c_1((L\mid_{X'})^p,h^pe^{-\varphi_z})\geq\lambda\Omega\,.\]
Consider the form $$g_{z,p}\in L^2_{n,1}(X',(L\mid_{X'})^p),\;g_{z,p}=\overline\partial\big(v_{z,p}\,\theta\big(\tfrac{|y-z|}{r_0}\big)e_\alpha^{\otimes p}\otimes e'_\alpha\big).$$
For simplicity, let $h_p$ also denote the metric induced on $(L\mid_{X'})^p\otimes\Lambda^{n,1}(T^*X')$ by $h$ and $\Omega$. Then
\[\int_{X'}\frac{1}{\lambda}\,|g_{z,p}|^2_{h_p}e^{-\varphi_z}\Omega^n=\int_{V}\frac{1}{\lambda}\,|g_{z,p}|^2_{h_p}e^{-\varphi_z}\Omega^n=\frac{1}{ap}\int_V|g_{z,p}|^2_{h_p}e^{-\varphi_z}\Omega^n<+\infty\,.\]
Note that the integral at the right is finite by \eqref{e:ohst}, since $\psi_\alpha(z)>-\infty$ and
$$\int_{V}|g_{z,p}|^2_{h_p}e^{-\varphi_z}\Omega^n\leq C'''\int_V|v_{z,p}|^2|\overline\partial\theta(\tfrac{|y-z|}{r_0})|^{2}e^{-2p\psi_\alpha}e^{-\varphi_z}\Omega^n\leq C''\int_V|v_{z,p}|^2e^{-2p\psi_\alpha}\Omega^n,$$
where $C''',C''>0$ are constants that depend only on $x$.

\par By the hypotheses of Theorem \ref{T:mt1}, $X'$ carries a complete K\"ahler metric. Thus, the hypotheses of Corollary \ref{l2est2} are satisfied for the K\"ahler manifold $(X',\Omega)$, the semipositive line bundle $((L\mid_{X'})^p,h^pe^{-\varphi_z})$ and the form $g_{z,p}$\,, for all $p\geq p_0$ and $z\in U$. So there exists $u_{z,p}\in L^2_{n,0}(X',(L\mid_{X'})^p)$ such that $\overline\partial u_{z,p} =g_{z,p}$ and
\begin{equation}\label{e:cs1}
\int_{X'}|u_{z,p}|^2_{h_p}e^{-\varphi_z}\,\Omega^n\leq\int_{X'}\frac{1}{\lambda}\,|g_{z,p}|^2_{h_p}e^{-\varphi_z}\Omega^n\leq\frac{C''}{ap}\int_V|v_{z,p}|^2e^{-2p\psi_\alpha}\Omega^n\,.
\end{equation}
Near $z$, $e^{-\varphi_z(y)}=r_0^{2n}|y-z|^{-2n}$ is not integrable, thus \eqref{e:cs1} implies that $u_{z,p}(z)=0$. Define
$$S_{z,p}:=v_{z,p}\,\theta\big(\tfrac{|y-z|}{r_0}\big)e_\alpha^{\otimes p}\otimes e'_\alpha-u_{z,p}\,.$$
Then $\overline\partial S_{z,p}=0$, $S_{z,p}(z)=v_{z,p}(z)e_\alpha^{\otimes p}\otimes e'_\alpha(z)\neq0$, $S_{z,p}\in H^0_{(2)}(X',(L\mid_{X'})^p\otimes K_{X'})$. Since $\varphi_z\leq0$ on $X'$, we have by \eqref{e:cs1} and \eqref{e:ohst}
\begin{eqnarray*}
\|S_{z,p}\|^2=\int_{X'}|S_{z,p}|^2_{h_p}\,\Omega^n&\leq&2\left(C'''\int_V|v_{z,p}|^2e^{-2p\psi_\alpha}\Omega^n+\int_{X'}|u_{z,p}|^2_{h_p}e^{-\varphi_z}\,\Omega^n\right)\\
&\leq&2C'\left(C'''+\frac{C''}{ap}\right)|v_{z,p}(z)|^2e^{-2p\psi_\alpha(z)}\leq C_1|S_{z,p}(z)|^2_{h_p},
\end{eqnarray*}
with a constant $C_1>0$ that depends only on $x$. This concludes the proof of \eqref{e:sp}.
\end{proof}

\medskip

\begin{proof}[Proof of Theorem \ref{T:mt1}] Let us write $\gamma_p=\{\widetilde\gamma_{p,i}\}$, where $\widetilde\gamma_{p,i}$ is the corresponding Fubini-Study current on $\widetilde U_i$ defined as in \eqref{e:FSc}. We fix $x\in G$ and let $(\widetilde U_i,\Gamma_i,\phi_i)$ be an orbifold chart so that $x\in U_i\subset G$ and $L_{\widetilde U_i}$, $K_{\widetilde U_i}$ have holomorphic frames on $\widetilde U_i$. Set $A_i=U_i\cap X_{sing}^{orb}$ and $\widetilde A_i=\phi_i^{-1}(A_i)$.

\smallskip

\par $(i)$ Lemma \ref{L:symp} (and its proof) shows that there exist psh functions $v_p,\,u$ and a continuous function $\rho$ on $U_i$ so that $\rho\circ\phi_i$ is smooth, 
$$\widetilde\gamma_{p,i}=dd^c(v_p\circ\phi_i)\,,\;\;c_1(L_{\widetilde U_i},\widetilde h_i)=dd^c(u\circ\phi_i)\,,\;\;c_1(K_{\widetilde U_i},\widetilde h_i^{K_{\widetilde U_i}})=dd^c(\rho\circ\phi_i)\,,$$
and $2v_p-2(pu+\rho)=\log P_p$\,. Hence on $\widetilde U_i$,
$$\frac{1}{p}\,v_p\circ\phi_i-u\circ\phi_i=\frac{1}{2p}\,\log P_p\circ\phi_i+\frac{1}{p}\,\rho\circ\phi_i\,.$$

\par By Theorem \ref{T:Pp} we have $\frac{1}{p}\,\log P_p\circ\phi_i\to0$, so $\frac{1}{p}\,v_p\circ\phi_i\to u\circ\phi_i$, in $L^1_{loc}(\widetilde U_i\setminus\widetilde A_i)$. It follows that the sequence $\{\frac{1}{p}\,v_p\circ\phi_i\}$ is locally uniformly upper bounded in $\widetilde U_i\setminus\widetilde A_i$.

\par If $y\in\widetilde A_i$ we may assume that there exist coordinates $z$ on some neighborhood $V\subset\widetilde U_i$ of $y=0$ so that $V\cap\widetilde A_i$ is contained in the cone $\{|z_n|\leq\max(|z_1|,\dots,|z_{n-1}|)\}$. Applying the maximum principle on complex lines parallel to the $z_n$ axis, we see that there exist a neighborhood $V_1\subset V$ of $y$ and a compact set $K\subset V\setminus\widetilde A_i$ so that $\sup_{V_1}v_p\circ\phi_i\leq\sup_Kv_p\circ\phi_i$. Hence $\{\frac{1}{p}\,v_p\circ\phi_i\}$ is uniformly upper bounded on $V_1$.

\par We conclude that the sequence $\{\frac{1}{p}\,v_p\circ\phi_i\}$ is locally uniformly upper bounded in $\widetilde U_i$. Hence it is relatively compact in $L^1_{loc}(\widetilde U_i)$ and it converges to $u\circ\phi_i$ in $L^1_{loc}(\widetilde U_i)$, since it does so outside $\widetilde A_i$ (see \cite[Theorem 3.2.12]{Ho}). This implies that $\frac{1}{p}\,\widetilde\gamma_{p,i}\to c_1(L_{\widetilde U_i},\widetilde h_i)$ weakly on $\widetilde U_i$, which yields $(i)$.

\smallskip

$(ii)$ The proof of $(i)$ implies that $\frac{1}{p}\,\log P_p\circ\phi_i\to0$ in $L^1_{loc}(\widetilde U_i)$.
We may assume that there is an embedding $U_i\hookrightarrow{\mathbb C}^N$. If  $K\Subset U_i$ and $z$ denote the coordinates on ${\mathbb C}^N$ then
$$\int_K|\log P_p|\,(dd^c\|z\|^2)^n=\frac{1}{m_i}\,\int_{\phi_i^{-1}(K)}|\log P_p\circ\phi_i|\,(dd^c\|\phi_i\|^2)^n,$$
where $m_i=|\Gamma_i|$. Since $(dd^c\|\phi_i\|^2)^n$ is smooth, it follows that $\frac{1}{p}\,\log P_p\to0$ in $L^1_{loc}(U_i)$ with respect to the area measure of $X$.
\end{proof}

\par To prove Theorem \ref{T:mt2}, we will need the following theorem:

\begin{Theorem}\label{T:Pp1} In the setting of Theorem \ref{T:mt2}, we have that $\frac{1}{p}\,\log P_p\to0$ as $p\to\infty$, in $L_{loc}^1(G\cap X_{reg}^{orb},\Omega^n)$\,.
\end{Theorem}

\begin{proof} By Lemma \ref{L:Sext}, $P_p$ is the Bergman kernel function of $H_{(2)}^0(X',(L\mid_{X'})^p)=H_{(2)}^0({\mathcal X},L^p)$, where the norm is denoted by $\|\cdot\|$.
We repeat the proof of Theorem \ref{T:Pp} by using the same notations and replacing $(n,q)$-forms with $(0,q)$-forms.
Hence the only non-formal difference is the proof of the analogue of lower estimate from \eqref{e:Bke}. More precisely, we have to show that there exist a constant $C_1>0$ and $p_0\in\mathbb{N}$ such that for all $p>p_0$ and all $z\in U$ with $\psi_\alpha(z)>-\infty$ there is a section $S_{z,p}\in H_{(2)}^0(X',(L\mid_{X'})^p)$ with $S_{z,p}(z)\neq0$ and satifying \eqref{e:sp}.

\par Let $v_{z,p}\in{\mathcal O}(V)$ with $v_{z,p}(z)\neq0$ satisfying \eqref{e:ohst}, given by the Ohsawa-Takegoshi theorem, and define the quasi-psh function $\varphi_z$ as in \eqref{e:weight}. We solve the $\overline\partial$-equation with $L^2$-estimates in order to extend $v_{z,p}(z)e_\alpha^{\otimes p}(z)$ to the desired section of $(L\mid_{X'})^p$ over $X'$.  To this end, we proceed exactly as in the proof of Theorem \ref{T:Pp}, and use Corollary \ref{l2est23} instead of Corollary \ref{l2est2}.
\end{proof}

\begin{proof}[Proof of Theorem \ref{T:mt2}] We follow the proof of Theorem \ref{T:mt1} and apply Theorem \ref{T:Pp1} instead of Theorem \ref{T:Pp}.
\end{proof}

\begin{proof}[Proof of Theorem \ref{T:mt3}] Since $(L,{\mathcal X},\widetilde h)$ is a positive orbifold line bundle, \cite[Theorem 5.4.19]{MM07} yields $\frac{1}{p}\,\log P_p\to0$ in $L_{loc}^1(X_{reg}^{orb},\Omega^n)$ as $p\to\infty$, so the proof of Theorem \ref{T:mt1} applies again.
\end{proof}
In the case of Theorem \ref{T:mt3} the convergence of the induced Fubini-Study metric to the initial metric is quite explicit. Namely, it is shown in \cite[Theorem 5.4.19]{MM07} that there exists $c>0$
such that for any $\ell\in\mathbb{N}$, there exists $C_\ell>0$ with 
\begin{equation*}\label{pb4.45}
\Big|\frac{1}{p}(\Phi_p^\ast \omega_{FS})(x) -\omega(x)\Big|_{\mathscr{C}^\ell}\leqslant C_\ell
\Big(\frac{1}{p}+ p^{\ell/2} \exp\big(-c\sqrt{p}\, d(x,X^{orb}_{sing})\big)\Big)\,,\:\: p\gg1.
\end{equation*}

We end the section with two lemmas which imply Proposition \ref{P:cKS}.

\begin{Lemma}
Let ${\mathcal X}=(X,{\mathcal U})$ be a compact K\"ahler orbifold. Then $X_{reg}^{orb}$ carries a complete K\"ahler metric.
\end{Lemma}
\begin{proof}
We will adapt the arguments of \cite{O87} to our context. Let $\omega$ be a K\"ahler current on $X$ whose local potentials are continuous near each $x\in X$ and smooth near each $x\in X_{reg}^{orb}$ (cf.\ Proposition \ref{P:Kc}). 
Let us consider a Hermitian metric $\eta$ on the complex space $X$ 
(
see e.\,g.\ \cite[Definition\,3.4.12]{MM07}).
Now, $X_{reg}^{orb}=X\setminus X_{sing}^{orb}$ and $X_{sing}^{orb}$ is an analytic set containing $X_{sing}$\,.
Repeating the proof from \cite[Proposition 1.1]{O87} we find a smooth proper function
$\psi:X_{reg}^{orb}\to (-\infty,0]$ and a constant $A_0>0$ such that for all $A>A_0$, $\eta_A=i\partial\overline\partial\psi+A \eta$ is a Hermitian metric on $X_{reg}^{orb}$ and the gradient of $\psi$ with respect to $\eta_A$ is bounded. Since $X$ is compact, there exists $\alpha>0$ such that $\omega\geq\alpha\eta$ on $X_{reg}$\,. Set  $\omega_A=i\partial\overline\partial\psi+(A/\alpha)\omega$, so $\omega_A\geq\eta_A$\,. Hence, for $A$ sufficiently large, $\omega_A$ is a K\"ahler metric and the gradient of $\psi$ with respect to $\omega_A$ is bounded, thus $\omega_A$ is complete.
\end{proof}

\begin{Lemma}
Let $X$ be a reduced Stein space of dimension $n$ and let $Y$ be a (closed) analytic subset of $X$ containing the singular locus $X_{sing}$\,. Then $X\setminus Y$ carries a complete K\"ahler metric.
\end{Lemma}

\begin{proof}
A result of Narasimhan \cite[Theorem 5]{Nar60} states that a reduced Stein space $X$ of dimension $n$ admits a holomorphic mapping $\iota:X\to\mathbb{C}^{2n+1}$, which is one-to-one, proper and regular on $X_{reg}$\,. Hence $X\setminus Y$ is biholomorphic to $\iota(X)\setminus\iota(Y)$. Since $\iota$ is proper, $\iota(Y)$ is an analytic subset of $\mathbb{C}^{2n+1}$. 
By a theorem of Grauert
\cite[Satz\,A,\,p.\,51]{Gra56}, $\mathbb{C}^{2n+1}\setminus\iota(Y)$ admits a complete K\"ahler metric. Its restriction to the closed submanifold $\iota(X)\setminus\iota(Y)\cong X\setminus Y$
is a complete K\"ahler metric.
%
\end{proof}
\begin{Example}
Let us state some interesting particular cases of our results.

\par $(i)$ Let $X$ be a Stein manifold of dimension $n$. Let $\varphi$ be a psh function on $X$ which is strictly psh on an open set $G$. Consider the Hilbert spaces
\[H^{n,0}_{(2)}(X,p\varphi)=\Big\{\text{$s$ holomorphic $n$-form}:\int_X i^{n^2}s\wedge\overline{s}\,e^{-2p\,\varphi}<\infty\Big\}\,.\]
Thus $H^{n,0}_{(2)}(X,p\varphi)=H^{n,0}_{(2)}(X,L^p)$, where $(L,h)$ is the trivial line bundle endowed with the metric $h=e^{-\varphi}$. 
Let $\{S_j^p\}$ be an orthonormal basis of $H^{n,0}_{(2)}(X,p\varphi)$. Let $\Omega^n$ be a volume element on $X$ and define $f_j^p=i^{n^2}S_j^p\wedge\overline{S_j^p}/\Omega^n$. The Fubini-Study current associated to $H^{n,0}_{(2)}(X,p\varphi)$ is 
$\gamma_p=\frac12 dd^c\log(\sum_{j=1}^\infty|f_j^p|^2)$. By Theorem \ref{T:mt1} and Proposition \ref{P:cKS} (ii), we have 
$\frac{1}{p}\,\gamma_p\to dd^c\varphi$ weakly as currents on $G$ as $p\to\infty$.

\par $(ii)$ Let $X$ be a Stein space of pure dimension $n$. Assume that $X$ has only orbifold singularities, i.\,e.\ $X$ admits an orbifold structure such that $X_{sing}^{orb}=X_{sing}$\,. Let $\varphi$ be a psh function on $X$ which is strictly psh on an open set $G$. Consider the Hilbert spaces
\[H^{n,0}_{(2)}(X,p\varphi)=\Big\{\text{$s$ holomorphic $n$-form on $X_{reg}$}:\int_{X_{reg}} i^{n^2}s\wedge\overline{s}\,e^{-2p\,\varphi}<\infty\Big\}\,.\]
By defining the Fubini-Study currents as above with the help of a volume element $\Omega^n$ on $X_{reg}$ we obtain that   
$\frac{1}{p}\,\gamma_p\to dd^c\varphi$ weakly as currents on $G_{reg}=G\cap X_{reg}$ as $p\to\infty$.


\par $(iii)$ Let $X$ be a complete K\"ahler manifold, $(L,h)\longrightarrow X$ a holomorphic line bundle endowed with a singular Hermitian metric $h$ with $c_1(L,h)\geq0$ in the sense of currents. Assume that $c_1(L,h)$ is strictly positive on an open set $G$. Denote by $\gamma_p$ the Fubini-Study current associated to $H^{n,0}_{(2)}(X,L^p)$. Then, as $p\to\infty$, $\frac{1}{p}\,\log P_p\to 0$ in $L^1_{loc}(G)$ and $\frac{1}{p}\,\gamma_p\to c_1(L,h)$ weakly as currents on $G$. Note that actually the full asymptotic expansion of the Bergman kernel \eqref{e:full_asy} on $G$ was obtained in \cite{HsM11}, in the case that $h$ is smooth. 
\end{Example}

\section{Distribution of zeros of random sections}\label{S:SZ}

\par In this section we prove Theorem \ref{T:SZ}. For the convenience of the reader, we recall here the results of Shiffman and Zelditch \cite{ShZ99,ShZ08} that are needed to prove Theorem \ref{T:SZ}. As noted in \cite{ShZ08}, their calculations hold in a general setting. 

\medskip

\par Let $(L,{\mathcal X},\widetilde h)$ be an orbifold line bundle over the $n$-dimensional orbifold $\mathcal X$, endowed with a singular Hermitian metric $\widetilde h$, and let $\widetilde\Omega$ be a Hermitian form on $\mathcal X$. No further assumptions on $\mathcal X,\widetilde\Omega,\widetilde h$ are made at this time.

\par If $S=\{\widetilde S_i\}\in H^0(\mathcal X,L)$, $S\neq0$, we denote by $[\widetilde S_i=0]$ the current of integration (with multiplicities) over the analytic hypersurface $\{\widetilde S_i=0\}\subset\widetilde U_i$, and we set $[S=0]=\{[\widetilde S_i=0]\}\in\widetilde{\mathcal T}$.

\smallskip

\subsection{Expectation and variance estimate}

\par Suppose $\mathcal V\subset H^0_{(2)}({\mathcal X},L)$ is a {\em finite dimensional} vector subspace of $L^2$-holomorphic sections, and let $\{S_1,\dots,S_d\}$ be an orthonormal basis of $\mathcal V$, where $d=\dim\mathcal V$. The Bergman kernel function $P$ and Fubini-Study current $\widetilde\alpha$ are defined as in Section \ref{SS:FS} (see \eqref{e:Bergfcn} and \eqref{e:FSc}). Note that, for every orbifold chart $(\widetilde U_i,\Gamma_i,\phi_i)$, $\log P\circ\phi_i\in L^1_{loc}(\widetilde U_i,\widetilde\Omega_i^n)$ since it is locally the difference of a psh function and the weight of the metric $\widetilde h_i$.

\par Following the framework in \cite{ShZ99}, we identify the unit sphere $\mathcal S$ of $\mathcal V$ to the unit sphere ${\mathbf S}^{2d-1}$ in ${\mathbb C}^d$ by
$$a=(a_1,\dots,a_d)\in{\mathbf S}^{2d-1}\longmapsto S_a=\sum_{j=1}^da_jS_j\in{\mathcal S},$$
and we let $\lambda$ be the probability measure on $\mathcal S$ induced by the normalized surface measure on ${\mathbf S}^{2d-1}$, denoted also by $\lambda$ (i.e. $\lambda({\mathbf S}^{2d-1})=1$).

\par Consider the function
$$Y(a)=\big\langle[S_a=0]-\widetilde\alpha,\widetilde\theta\,\big\rangle\,,\;\;a\in{\mathbf S}^{2d-1}\,,$$
where $\widetilde\theta$ is a $(n-1,n-1)$ test form on $\mathcal X$. 
Following \cite[Sec. 3.1, 3.2]{ShZ99} (see also \cite[Sec.\,5.3]{MM07}) we obtain in our setting
\begin{equation}
\int_{{\mathbf S}^{2d-1}}Y(a)\,d\lambda(a)=0\,,  \label{e:SZ1}
\end{equation}
(cf.\ \cite[Theorem 5.3.1]{MM07}) and by using \eqref{e:SZ1} (cf.\ \cite[Lemma 5.3.2, (5.3.22-23)]{MM07}), 
\begin{equation}
\int_{{\mathbf S}^{2d-1}}|Y(a)|^2\,d\lambda(a)\leq A\,C_{\widetilde\theta}\,,\,\text{ with } A=\frac{1}{\pi^2}\int_{{\mathbb C}^2}(\log|z_1|)^2e^{-|z_1|^2-|z_2|^2}\,dz\,,  \label{e:SZ2}
\end{equation}
where $dz$ is the Lebesgue measure on ${\mathbb C}^2$, and $C_{\widetilde\theta}\,$ is a constant depending only on $\widetilde\theta$. Equation \eqref{e:SZ1} says that the expectation of the current-valued random variable $a\mapsto[S_a=0]$ is the Fubini-Study current $\widetilde\alpha$.

\smallskip

\subsection{Almost everywhere convergence}

\par We assume now that $(\mathcal X,\widetilde\Omega)$ is a {\em compact K\"ahler} orbifold and $(E_p,\mathcal X,\widetilde h_p)$ are singular Hermitian orbifold line bundles with the following property: there exists a constant $C>0$ so that $\big\langle c_1(E_p,\widetilde h_p),\widetilde\Omega^{n-1}\big\rangle\leq Cp$, for all $p>0$.

\par Let $\mathcal V^p\subset H^0_{(2)}({\mathcal X},E_p)$ be vector subspaces of $L^2$-holomorphic sections with corresponding Fubini-Study currents $\widetilde\alpha_p$. Consider the unit sphere ${\mathcal S}^p\subset{\mathcal V}^p$ with the probability measure $\lambda_p$ induced as above by the normalized area measure on the unit sphere ${\mathbf S}^{2d_p-1}$ in ${\mathbb C}^{d_p}$ via the identification ${\mathcal S}^p\equiv{\mathbf S}^{2d_p-1}$, where $d_p=\dim{\mathcal V}^p$. Finally, consider the probability space ${\mathcal S}_\infty=\prod_{p=1}^\infty{\mathcal S}^p$ endowed with the probability measure $\lambda_\infty=\prod_{p=1}^\infty\lambda_p$. Recall that formulas \eqref{e:SZ1} and \eqref{e:SZ2} hold for each of the spaces $\mathcal V^p$. 

\par 
In analogy to \cite[Sec.\ 3.3]{ShZ99} (see also \cite[Sec.\,5.3]{MM07}) we see that
for $\lambda_\infty$-a.e. sequence $\{\sigma_p\}_{p\geq1}\in{\mathcal S}_\infty$,
\begin{equation}\label{e:SZ3}
\lim_{p\to\infty}\,\frac{1}{p}\,\big([\sigma_p=0]-\widetilde\alpha_p\big)=0\,,
\end{equation}
weakly in the sense of currents on $\mathcal X$. The argument is as follows. Since
$$\big\langle[\sigma_p=0],\widetilde\Omega^{n-1}\big\rangle=\big\langle\widetilde\alpha_p,\widetilde\Omega^{n-1}\big\rangle=\big\langle c_1(E_p,\widetilde h_p),\widetilde\Omega^{n-1}\big\rangle\leq Cp\,,$$
(this follows from the Poincar\'e-Lelong formula \cite[Theorem 2.3.3]{MM07}) it suffices to show that, for a fixed test form $\widetilde\theta$, one has
$$\lim_{p\to\infty}\,\frac{1}{p}\,\big\langle[\sigma_p=0]-\widetilde\alpha_p,\widetilde\theta\,\big\rangle=0\,,$$
for $\lambda_\infty$-a.e. $\sigma=\{\sigma_p\}_{p\geq1}\in{\mathcal S}_\infty$. Setting
$$Y_p:{\mathcal S}_\infty\longrightarrow\mathbb{C}\,,\quad Y_p(\sigma)=\frac{1}{p}\,\big\langle[\sigma_p=0]-\widetilde\alpha_p,\widetilde\theta\,\big\rangle\,,$$
it follows by \eqref{e:SZ2} that
$$\int_{{\mathcal S}_\infty}\left(\sum_{p=1}^\infty|Y_p|^2\right)\,d\lambda_\infty=\sum_{p=1}^\infty\int_{{\mathcal S}^p}|Y_p|^2\,d\lambda_p\leq A\,C_{\widetilde\theta}\;\sum_{p=1}^\infty\frac{1}{p^2}<+\infty\,,$$
which yields the desired conclusion.

\begin{proof}[Proof of Theorem \ref{T:SZ}]
This follows from \eqref{e:SZ3} and from Theorems \ref{T:mt1}, \ref{T:mt2}, \ref{T:mt3}, and Corollary \ref{C:mt1}.
\end{proof}


\begin{thebibliography}{XXXXX}

\bibitem[ALR]{ALR07} A. Adem, J. Leida and Y. Ruan, {\em Orbifolds and Stringy Topology}, Cambridge Tracts in Mathematics, vol. 171, Cambridge University Press, 2007, 149 p.









\bibitem[BG]{BG08} C. Boyer and K. Galicki, {\em Sasakian Geometry}, Oxford Mathematical Monographs, Oxford University Press, 2008, 631 p.

\bibitem[BGK]{BGK05} C. Boyer, K. Galicki and J. Koll\'ar, {\em Einstein metrics on spheres}, Ann. of Math. (2) {\bf 162} (2005), 557--580.

\bibitem[Ca]{Ca99} D. Catlin, {\em The Bergman kernel and a theorem of Tian}, in {\em Analysis and geometry in several complex variables (Katata, 1997)}, 1--23, Trends Math., Birkh\"auser, Boston, 1999.



\bibitem[CM]{CM11} D. Coman and G. Marinescu, {\em Equidistribution results for singular metrics on line bundles}, preprint available at arXiv:1108.5163, 2011.

\bibitem[DLM1]{DLM04a} X.~Dai, K.~Liu, and X.~Ma, \emph{On the asymptotic expansion of {B}ergman kernel}, J.\ Differential Geom.\ \textbf{72} (2006), 1--41; announced in C. R. Math. Acad. Sci. Paris \textbf{339} (2004), 193--198.

\bibitem[DLM2]{DLM012} X.~Dai, K.~Liu, and X.~Ma, \emph{A remark on weighted Bergman kernels on orbifolds}, Math.\ Res.\ Lett.\ \textbf{19} (2012), 143--148.


\bibitem[D1]{D82} J.\ P.\ Demailly, {\em Estimations $L^2$ pour l'op\'erateur $\overline\partial$ d'un fibr\'e holomorphe semipositif au--dessus d'une vari\'et\'e k\"ahl\'erienne compl\`ete}, Ann.\ Sci.\ \'Ecole Norm.\ Sup.\ {\bf 15} (1982), 457--511.

\bibitem[D2]{D85} J.\ P.\ Demailly, {\em Mesures de Monge-Amp\`ere et caract\'erisation g\'eom\'etrique des vari\'et\'es alg\'ebriques affines}, M\'em.\ Soc.\ Math.\ France (N.S.) No. 19 (1985), 1--125.

\bibitem[D3]{D90} J.\ P.\ Demailly, {\em Singular Hermitian metrics on positive line bundles}, in {\em  Complex algebraic varieties (Bayreuth, 1990)}, Lecture Notes in Math. 1507, Springer, Berlin, 1992, 87--104.

\bibitem[D4]{D92} J.\ P.\ Demailly, {\em Regularization of closed positive currents and intersection theory}, J.\ Algebraic Geom.\ {\bf 1} (1992), 361--409.


\bibitem[D5]{D93b} J.\ P.\ Demailly, {\em A numerical criterion for very ample line bundles}, J.\ Differential Geom.\ {\bf 37} (1993), 323--374.


\bibitem[DMS]{DMS} T. C. Dinh, G. Marinescu and V. Schmidt, {\em Asymptotic distribution of zeros of holomorphic sections in the non compact setting}, J.\ Stat.\ Phys.\ \textbf{148} (2012), no.\ 1, 113-136.


\bibitem[DS]{DS06b} T. C. Dinh and N. Sibony, {\em Distribution des valeurs de transformations m\'eromorphes et applications}, Comment.\ Math.\ Helv.\ {\bf 81} (2006), 221--258.


\bibitem[D]{Don01}
S. K. Donaldson, \emph{Scalar curvature and projective embeddings. I.}, J.\ Differential Geom.\ \textbf{59} (2001), no.\ 3, 479--522.

\bibitem[EGZ]{EGZ09} P. Eyssidieux, V. Guedj and A. Zeriahi, {\em Singular K\"ahler-Einstein metrics},  J.\ Amer.\ Math.\ Soc.\ {\bf 22} (2009), 607--639.



\bibitem[FS]{FS95} J. E. Forn\ae ss and N. Sibony, {\em Oka's inequality for currents and  applications}, Math.\ Ann.\ {\bf 301} (1995), 399--419.



\bibitem[GK]{GK07} A. Ghigi and J. Koll\'ar, {\em K\"ahler-Einstein metrics on orbifolds and Einstein metrics on spheres}, Comment.\ Math.\ Helv.\ {\bf 82} (2007), 877--902.

\bibitem[G]{Gra56}
H.\ Grauert, \emph{Charakterisierung der Holomorphiegebiete durch die vollständige Kählersche Metrik}, Math.\ Ann.\ \textbf{131} (1956), 38--75.





\bibitem[Ho]{Ho} L.\ H\"ormander, {\em Notions of Convexity},  Reprint of the 1994 edition, Modern Birkh\"auser Classics, Basel, Birkh\"auser, 2007, viii, 414~p.

\bibitem[HsM]{HsM11} C.-Y.\ Hsiao and G. Marinescu, {\em The asymptotics for Bergman kernels for lower
energy forms and the multiplier ideal Bergman kernel asymptotics}, preprint available at arXiv:1112.5464, to appear in Comm.\ Anal.\ Geom.



\bibitem[MM1]{MM04} X. Ma and G. Marinescu, {\em Generalized Bergman kernels on symplectic manifolds},  C.\ R.\ Math.\ Acad.\ Sci.\ Paris {\bf 339} (2004), 493--498.

\bibitem[MM2]{MM07} X. Ma and G. Marinescu, {\em Holomorphic Morse Inequalities and Bergman Kernels}, Progress in Math., vol. 254, Birkh\"auser, Basel, 2007, xiii, 422 p.

\bibitem[MM3]{MM08} X. Ma and G. Marinescu, {\em Generalized Bergman kernels on symplectic manifolds}, Adv.\ Math.\ {\bf 217} (2008), 1756--1815.






\bibitem[N]{Nar60}
R.\ Narasimhan, \emph{Imbedding of holomorphically complete complex spaces}, Amer.\ J.\ Math.\ \textbf{82} (1960), 917--934. 


\bibitem[NV]{NoVo:98}
S.~Nonnenmacher and A.~Voros, \emph{Chaotic eigenfunctions in phase space}, J.\ Stat.\ Phys.\ \textbf{92} (1998), no.~3-4, 451--518.

\bibitem[O]{O87} T. Ohsawa, {\em Hodge spectral sequence and symmetry on compact K\"ahler spaces},  Publ.\ Res.\ Inst.\ Math.\ Sci.\ {\bf 23} (1987), 613--625.

\bibitem[OT]{OT87} T. Ohsawa and K. Takegoshi, {\em On the extension of $L^2$ holomorphic functions}, Math.\ Z.\ {\bf 195} (1987), 197--204.

\bibitem[RT]{RT11} 
J.\ Ross, and R.\ Thomas, 
\emph{Weighted projective embeddings, stability of orbifolds, and constant scalar curvature K\"ahler metrics},
J.\ Differential Geom.\ \textbf{88} (2011), no.\ 1, 109--159.

\bibitem[R]{Ru98} W. D. Ruan, {\em Canonical coordinates and Bergmann metrics}, Comm.\ Anal.\ Geom.\ {\bf 6} (1998), 589--631.



\bibitem[SZ1]{ShZ99} B. Shiffman and S. Zelditch, {\em Distribution of zeros of random and quantum chaotic sections of positive line bundles}, Comm.\ Math.\ Phys.\ {\bf 200} (1999), 661--683.

\bibitem[SZ2]{ShZ04} B. Shiffman and S. Zelditch, {\em Random polynomials with prescribed Newton polytope}, J.\ Amer.\ Math.\ Soc.\ {\bf 17} (2004), 49--108.

\bibitem[SZ3]{ShZ08} B. Shiffman and S. Zelditch, {\em  Number variance of random zeros on complex manifolds}, Geom.\ Funct.\ Anal.\ {\bf 18} (2008), 1422--1475.



\bibitem[T]{Ti90} G. Tian, {\em On a set of polarized K\"ahler metrics on algebraic manifolds}, J.\ Differential Geom.\ {\bf 32} (1990), 99--130.



\bibitem[Z]{Z98} S. Zelditch, {\em Szeg\"o kernels and a theorem of Tian}, Internat.\ Math.\ Res.\ Notices {\bf 1998}, no. 6, 317--331.

\bibitem[Y]{Yau87}
S.\ T.\ Yau, \emph{Nonlinear analysis in geometry}, Enseign.\ Math.\ (2) \textbf{33}
  (1987), no.~1-2, 109--158.


\end{thebibliography}
\end{document}